\documentclass[12pt]{amsart}

\usepackage[matrix,arrow,curve,frame]{xy}
\usepackage{amsmath,amsthm,amssymb,enumerate}
\usepackage{latexsym}
\usepackage{amscd}
\usepackage[colorlinks=false]{hyperref}
\usepackage{euscript}

\newtheorem{main}[subsection]{Main Theorem}
\newtheorem{thm}[subsection]{Theorem}

\newtheorem{cor}[subsection]{Corollary}

\newtheorem{conj}[subsection]{Conjecture}
\newtheorem{remark}[subsection]{Remark}

\theoremstyle{definition}

\numberwithin{equation}{section}

\def\cT{{\cal T}}

\def\cG{{\cal G}}

\def\cA{{\cal A}}

\def\ra{\rightarrow}

\def\cA{{\mathcal A}}

\def\cC{{\mathcal C}}

\def\cG{{\mathcal G}}
\def\cH{{\mathcal H}}
\def\cI{{\mathcal I}}

\def\cS{{\mathcal S}}
\def\cT{{\mathcal T}}

\def\cV{{\mathcal V}}
\def\cW{{\mathcal W}}

\def\gd{{\mathfrak d}}

\def\gg{{\mathfrak g}}

\def\gl{{\mathfrak l}}

\def\go{{\mathfrak o}}
\def\gp{{\mathfrak p}}

\def\gs{{\mathfrak s}}

\newfont{\german}{eufm10}

\allowdisplaybreaks
\begin{document}
\pagestyle{plain}

\title
{$N=4$ superconformal algebras and diagonal cosets}

\author{Thomas Creutzig}

\author{Boris Feigin} 

\author{Andrew R. Linshaw}

\address{Department of Mathematical and Statistical Sciences, University of Alberta, Edmonton, Alberta  T6G 2G1, Canada.}
\email{creutzig@ualberta.ca}

\address{National Research University Higher School of Economics, Russian Federation, International Laboratory of Representation Theory, Mathematical Physics, Russia, Moscow, 101000, Myasnitskaya ul., 20, Landau Institute for Theoretical Physics, Russia, Chernogolovka, 142432, pr. Akademika Semenova, 1a.}
\email{borfeigin@gmail.com}

\address{Department of Mathematics, University of Denver, Denver, CO 80208, U.S.A.}
\email{andrew.linshaw@du.edu}


{\abstract \noindent  Coset constructions of $\cW$-algebras have many applications, and were recently given for principal $\cW$-algebras of $A$, $D$, and $E$ types by Arakawa together with the first and third authors. In this paper, we give coset constructions of the large and small $N=4$ superconformal algebras, which are the minimal $\cW$-algebras of $\gd(2,1;a)$ and $\gp\gs\gl(2|2)$, respectively. From these realizations, one finds a remarkable connection between the large $N=4$ algebra and the diagonal coset $C^{k_1, k_2} = \text{Com}(V^{k_1+k_2}(\gs\gl_2), V^{k_1}(\gs\gl_2) \otimes V^{k_2}(\gs\gl_2))$, namely, as two-parameter vertex algebras, $C^{k_1, k_2}$ coincides with the coset of the large $N=4$ algebra by its affine subalgebra. We also show that at special points in the parameter space, the simple quotients of these cosets are isomorphic to various $\cW$-algebras. As a corollary, we give new examples of strongly rational principal $\cW$-algebras of type $C$ at degenerate admissible levels.}

\keywords{vertex algebra; $\cW$-algebra; $N=4$ superconformal algebra; coset construction}
\maketitle

\section{Introduction} \label{section:intro}

The $N=4$ superconformal algebra $V^{k, a}$ is to us the most exciting family of vertex superalgebras. It is a two-parameter vertex algebra depending on the complex numbers $k$ and $a$, and in the large $a$ limit it becomes the small $N=4$ superconformal algebra $V^k$ times a commutative vertex algebra. Both the small and large $N=4$ superconformal algebra are central in the physics of string theory, especially in the context of the AdS$_3$/CFT$_2$-correspondence \cite{Mal}. 

The $N=4$ superconformal algebra is special in many ways. It is the only example of a minimal $\cW$-superalgebra that is a two-parameter vertex algebra. The simple quotient of $V^{1/2, a}$ appears as a quantum geometric Langlands kernel vertex algebra in the context of $S$-duality for $SU(2)$ \cite{CGL}. Furthermore, the small $N=4$ superconformal algebra has a large group of outer automorphisms, namely $\text{SL}_2$. The simple quotient of $V^{-2}$ is the algebra of global sections of $K3$ surfaces and its $\mathbb{Z}_2$-orbifold is the one of Enriques surfaces \cite{Song}, see also \cite{CLR}.  The simple quotient of $V^{-2}$ is also the vertex superalgebra of Mathieu moonshine \cite{EOT}.
The Lie superalgebras $\gd(2, 1, a)$ for complex $a$ form a continuous family of $9+8$ dimensional Lie superalgebras and $V^{k, a}$ is the minimal $\cW$-algebra of $\gd(2, 1, a)$ at level $k$. Obviously it is interesting to gain a good understanding of vertex superalgebras associated to $\gd(2, 1, a)$.

Given a vertex algebra $V$ and a subalgebra $W\subseteq V$, the commutant $C=\text{Com}(W, V)$ is called a coset vertex algebra. 
Realizing $\cW$-algebras as coset vertex algebras is not only interesting in its own right but also has various applications. Recently and jointly with Tomoyuki Arakawa, two of us have proven that principal $\cW$-algebras of simply laced Lie algebras can be realized as coset vertex algebras \cite{ACL}. This result had various important consequences such as strong rationality of the cosets $\text{Com}(L_n(\gg\gl_m), L_n(\gs\gl_{m+1}))$ and $\text{Com}(L_{2n}(\gs\go_m), L_{2n}(\gs\go_{m+1}))$ for positive integers $n$ and $m$ \cite[Cor. 1.2 and 1.3]{ACL}. Here $L_k(\gg)$ denotes the simple quotient of the universal affine vertex algebra $V^k(\gg)$. Another consequence was rigidity of the vertex tensor categories of ordinary modules of affine vertex algebras of simply-laced Lie algebras at admissible level, and equivalences of these tensor categories with subcategories of modules of principal $\cW$-algebras \cite{C}. This illustrates the importance of coset realizations, namely in this case $V$ and $W$ are affine vertex algebras and only the representation theory of the coset $C=\text{Com}(W, V)$ as a strongly rational principal $\cW$-algebra was well understood \cite{A2, A3, A4}. Via the coset relation, the category of ordinary modules of the affine vertex algebra inherited important structure such as rigidity. Two more examples that illustrate the power of coset realizations for the representation theory of vertex superalgebras are the affine vertex superalgebra of $\go\gs\gp(1|2)$ at admissible level \cite{CFK, CKLR} and the $N=2$ superconformal algebra \cite{CLRW, Sato, Semikhatov, Adamovic}. The coset $\text{Com}(L_{k}(\gs\gl_2), L_{k}(\go\gs\gp(1|2)))$ at a positive integer or admissible level is a strongly rational Virasoro algebra, so $L_{k}(\go\gs\gp(1|2))$ can be realized as an extension of $L_{k}(\gs\gl_2)$ times the Virasoro algebra. This implies the strong rationality of  $L_{k}(\go\gs\gp(1|2))$ at positive integer or admissible levels, and the representation theory of $L_{k}(\go\gs\gp(1|2))$ is then inherited using the theory of \cite{CKM}. In the case of the $N=2$ superconformal algebra this goes the other way around, as this superalgebra is a coset of 
$L_{k}(\gs\gl_2)$ times a rank one $bc$ vertex superalgebra. Nonetheless, the representation theory of the $N=2$ superconformal algebra can be nicely related to the one of $L_{k}(\gs\gl_2)$ \cite{CLRW, Sato, Semikhatov, Adamovic}. 

In this work, we not only realize $V^{k,a}$ and $V^k$ as cosets but we also show that their cosets by their affine vertex subalgebras are diagonal cosets of $\gs\gl_2$-type in the case of $V^{k, a}$ and in the case of $V^k$ one has to take in addition an $\text{SL}_2$-orbifold. 
We are optimistic that our coset realizations of $V^{k, a}$ and $V^k$ will allow us to study the representation theory of their simple quotients  at special values of $k$ and $a$. Let us  note that the idea that vertex algebas associated to $\gd(2, 1, a)$ might be related to $\gs\gl_2$-cosets appeared a while ago in work of Semikhatov and one of us \cite{FS}, see also \cite{BFST}.

\subsection{Results}

By a two-parameter vertex algebra we mean a vertex algebra over a localization of the polynomial ring in two variables, and we are concerned with four such two-parameter families. First, the large $N=4$ superconformal algebra $V^{k, a}$ that is defined as the minimal $\cW$-algebra of $\gd(2, 1; a)$ at level $k$. Our point of view is that $V^{k, a}$ is a two-parameter vertex algebra over a localization of $\mathbb C[k, a]$, and $V^{k, a}$ contains as subalgebra the affine vertex algebra $V^\ell(\gs\gl_2) \otimes V^{\ell'}(\gs\gl_2)$ at levels $\ell = - \frac{a+1}{a}k-1$ and $\ell' = -(a+1)k-1$, respectively.  Its coset is denoted by $D^{k,a} := \text{Com}(V^{\ell}(\gs\gl_2) \otimes V^{\ell'}(\gs\gl_2), V^{k,a})$. Our third two-parameter vertex algebra is the diagonal coset $$C^{k_1, k_2}(\gs\gl_2) := \text{Com}\left(V^{k_1+k_2}(\gs\gl_2), V^{k_1}(\gs\gl_2) \otimes V^{k_2}(\gs\gl_2)\right).$$ The last two-parameter vertex algebra that we consider is denoted by $Y(\lambda, \mu)$. It is constructed in section \ref{sec:N4coset}, and its even subalgebra is the coset $$Y(\lambda, \mu)_0 = \text{Com}\left(L_1(\gs\gl_2) \otimes L_1(\gs\gl_2),L_1(\gd(2, 1;-\lambda) \otimes L_1(\gd(2, 1;-\mu)\right).$$
As usual, by $L_k(\gg)$ we mean the simple quotient of $V^k(\gg)$. 
These two-parameter vertex algebras enjoy a remarkable connection, see Theorem \ref{thm:N4}  and Corollary \ref{cor:cosetiso}:
\begin{main}
As two-parameter vertex algebras
\[
V^{-\frac{\lambda\mu}{\lambda+\mu}, \frac{\mu}{\lambda}} \cong  \text{Com}\left(V^{\lambda^{-1}+\mu^{-1}-2}(\gs\gl_2), Y(\lambda, \mu)\right) \qquad \text{and} \qquad 
D^{-\frac{\lambda\mu}{\lambda+\mu}, \frac{\mu}{\lambda}} \cong C^{\lambda^{-1}-1, \mu^{-1}-1}.
\]
\end{main}
A similar result holds for the small $N=4$ superconformal algebra as a one-parameter vertex algebra and these can be viewed as the $\mu\rightarrow \infty$ limit of the above results. For this, recall that in the large $a$ limit $\gd(2, 1; a)$ becomes $\gp\gs\gl(2|2) \oplus \mathbb C^3$ and the small $N=4$ superconformal algebra $V^k$ is defined as the minimal $\cW$-algebra of $\gp\gs\gl(2|2)$  at level $k$. It contains $V^{-k-1}(\gs\gl_2)$ as affine subalgebra, and we set $D^k := \text{Com}(V^{-k-1}(\gs\gl_2), V^k)$. Furthermore, we construct a one-parameter vertex algebra $Y(\lambda)$ whose even subalgebra is $Y(\lambda)_0 = \text{Com}\left(L_1(\gs\gl_2) \otimes L_1(\gs\gl_2),L_1(\gd(2, 1;-\lambda) \otimes L_1(\gp\gs\gl(2|2))\right).$ We prove in Theorem \ref{thm:smallN4} and Corollary \ref{cor:smallN4}
\begin{main}
As one-parameter vertex algebras 
\[
V^{-\lambda} \cong  \text{Com}\left(V^{\lambda^{-1} -2}(\gs\gl_2), Y(\lambda)\right) \qquad \text{and} \qquad
(D^{-\lambda})^{\text{SL}_2} \cong C^{\lambda^{-1}-1, -1}.
\]
\end{main}
The second part of the paper is then an investigation of $C^{k_1, k_2}$ and its simple quotients. First, in section \ref{sec:generic} we determine minimal strong generators for $C^{k_1, k_2}$.
\begin{main}
$C^{k_1, k_2}$ is of type $\cW(2,4,6,6,8,8,9,10,10,12)$ as a vertex algebra over a localization of the ring $\mathbb{C}[k_1, k_2]$. Equivalently, this holds for generic values of $k_1, k_2$.
\end{main}
This has been previously stated in the physics literature \cite{B-H,KS}, but to the best of our knowledge a rigorous proof has not previously appeared. 

Next, we study certain simple one-parameter quotients of $C^{k_1, k_2}$ by identifying them with quotients of the universal even spin $\cW_{\infty}$-algebra that was conjectured to exist in the physics literature \cite{H,CGKV}, and constructed in \cite{KL}. Recall that there is a unique two-parameter vertex algebra $\cW^{\text{ev}}(c,\lambda)$ that is freely generated of type $\cW(2,4,6,\dots)$, and is generated by a Virasoro field $L$ and a weight $4$ primary field $W^4$. All one-parameter vertex algebras of type $\cW(2,4,6,\dots, 2N)$ for some $N$ satisfying some mild hypotheses, arise as quotients of this universal algebra. This includes the principal $\cW$-algebras of type $C$ (and type $B$, by Feigin-Frenkel duality), and the $\mathbb{Z}_2$-orbifold of the principal $\cW$-algebras of type $D$, as well as many others arising as cosets of affine vertex algebras inside larger structures. One-parameter quotients of $\cW^{\text{ev}}(c,\lambda)$ are in bijection with a certain family of curves in the parameter space $\mathbb{C}^2$ called {\it truncation curves}, and nontrivial pointwise coincidences between these vertex algebras correspond to intersection points on these curves. Based on computations of Hornfeck \cite{H}, the explicit truncation curves for the $\cW$-algebras $\cW^k(\gs\gp_{2n}, f_{\text{prin}})$, and the orbifolds $\cW^k(\gs\go_{2n}, f_{\text{prin}})^{\mathbb{Z}_2}$ were written down in \cite{KL}, and the pointwise coincidences between the simple quotients of these algebras were classified; see Theorems 9.1, 9.3, and 9.4 of \cite{KL}.

Similarly, there are certain prime ideals $J = (p(k_1, k_2)) \subseteq \mathbb{C}[k_1, k_2]$ for which the quotient $C^{J, k_1, k_2} =  C^{k_1, k_2} / J \cdot C^{k_1, k_2}$ is not simple as a vertex algebra over the ring $\mathbb{C}[k_1, k_2] / J$. We shall denote the simple graded quotient $C^{J,k_1, k_2}$ by its maximal proper graded ideal by $C^{k_1, k_2}_J$.

\begin{main}
For $J = (k_2 -2)$, $J= (k_2 + 1/2)$, and $J = (k_1 - k_2)$, the simple, one-parameter vertex algebras $C^{k_1, k_2}_J$ can be identified with quotients of $\cW^{\text{ev}}(c,\lambda)$. 
\end{main}
In all these cases, we shall identify the explicit truncation curve that allows this identification. By finding the intersection points between these truncation curves and the truncation curves for the above $\cW$-algebras, we classify the points where the simple quotients $C_{k_1, 2}$, $C_{k_1, -1/2}$, and $C_{k_1, k_1}$, are isomorphic to $\cW_{\ell}(\gs\gp_{2n}, f_{\text{prin}})$ or $\cW_{\ell}(\gs\go_{2n}, f_{\text{prin}})^{\mathbb{Z}_2}$ for some $\ell$ and $n$. 

In the case $J = (k_2+1)$,  $C^{J, k_1, k_2}$ is already simple as a one-parameter vertex algebra, and we identify it with the $\mathbb{Z}_2$-orbifold of a quotient of the universal $\cW_{\infty}$-algebra $\cW(c,\lambda)$ which was constructed in \cite{L6}. Recall that  $\cW(c,\lambda)$ is freely generated of type $\cW(2,3,\dots)$, and is generated by a Virasoro field $L$ and a weight $3$ primary field $W^3$. Also, it has full automorphism group $\mathbb{Z}_2$. All one-parameter vertex algebras of type $\cW(2,3,\dots, N)$ for some $N$ satisfying some mild hypotheses, arise as quotients of this universal algebra. This includes the type $A$ principal $\cW$-algebras and many families of cosets such as the generalized parafermion algebras $\cG^{\ell}(n) = \text{Com}(V^k(\gg\gl_n), V^k(\gs\gl_{n+1}))$. Such one-parameter quotients are in bijection with a family of plane curves, and coincidences between these algebras correspond to intersection points on their curves. As a consequence, we classify all coincidences between $C_{k_1, -1}$ and the simple orbifolds $\cW_{\ell}(\gs\gl_n, f_{\text{prin}})^{\mathbb{Z}_2}$ and $\cG_{\ell}(n)^{\mathbb{Z}_2}$.

The strong rationality of principal $\cW$-algebras at nondegenerate admissible levels is a fundamental result of Arakawa \cite{A1,A2}, and it is of interest to find new examples of strongly rational principal $\cW$-algebras. Some infinite families of examples of types $B$, $C$, and $D$ principal $\cW$-algebras at nonadmissible levels that are strongly rational were found in \cite{KL}, but a complete classification of such algebras is currently out of reach. In this paper, we give new examples in type $C$ (and therefore type $B$) by combining our classification of isomorphisms $C_{k_1, 2} \cong \cW_{\ell}(\gs\gp_{2n}, f_{\text{prin}})$ with the independent observation that $C_{k_1, 2}$ is strongly rational for all positive integer or admissible levels $k_1$.

\begin{main}
At the degenerate admissible level $\displaystyle \ell = -(n+1) + \frac{2n-1}{4 (n-1)}$, $\cW_{\ell}(\gs\gp_{2n}, f_{\text{prin}})$ is strongly rational. These are new examples of strongly rational principal $\cW$-algebras.
\end{main}

\noindent {\bf Acknowledgements} T.C is supported by NSERC Discovery Grant \#RES0020460. The research of B. F. was carried out within the
HSE University Basic Research Program and funded (jointly) by Russian Academic Excellence Project \lq\lq 5-100" and by the Russian Science Foundation (project 16-11-10316). A. L. is supported by Simons Foundation Collaboration Grant \#318755.

\section{Background} \label{section:background}

We begin by defining vertex algebras, which have been discussed from various different points of view in the literature (see \cite{Bor,FLM,FHL,K,FBZ}, for example). We follow the formalism developed in \cite{LZ} and partly in \cite{Li}. Let $V=V_0\oplus V_1$ be a super vector space over $\mathbb{C}$, $z,w$ be formal variables, and $\text{QO}(V)$ be the space of linear maps $$V\ra V((z))=\{\sum_{n\in\mathbb{Z}} v(n) z^{-n-1}|
v(n)\in V,\ v(n)=0\ \text{for} \ n>\!\!>0 \}.$$ Each element $a\in \text{QO}(V)$ can be represented as a power series
$$a=a(z)=\sum_{n\in\mathbb{Z}}a(n)z^{-n-1}\in \text{End}(V)[[z,z^{-1}]].$$ We assume that $a=a_0+a_1$ where $a_i:V_j\ra V_{i+j}((z))$ for $i,j\in\mathbb{Z}_2$, and we write $|a_i| = i$.

For each $n \in \mathbb{Z}$, we have a nonassociative bilinear operation on $\text{QO}(V)$, defined on homogeneous elements $a$ and $b$ by
$$ a(w)_{(n)}b(w)=\text{Res}_z a(z)b(w)\ \iota_{|z|>|w|}(z-w)^n- (-1)^{|a||b|}\text{Res}_z b(w)a(z)\ \iota_{|w|>|z|}(z-w)^n.$$
Here $\iota_{|z|>|w|}f(z,w)\in\mathbb{C}[[z,z^{-1},w,w^{-1}]]$ denotes the power series expansion of a rational function $f$ in the region $|z|>|w|$. For $a,b\in \text{QO}(V)$, we have the following identity of power series known as the {\it operator product expansion} (OPE) formula.
 \begin{equation}\label{opeform} a(z)b(w)=\sum_{n\geq 0}a(w)_{(n)} b(w)\ (z-w)^{-n-1}+:a(z)b(w):. \end{equation}
Here $:a(z)b(w):\ =a(z)_-b(w)\ +\ (-1)^{|a||b|} b(w)a(z)_+$, where $a(z)_-=\sum_{n<0}a(n)z^{-n-1}$ and $a(z)_+=\sum_{n\geq 0}a(n)z^{-n-1}$. Often, \eqref{opeform} is written as
$$a(z)b(w)\sim\sum_{n\geq 0}a(w)_{(n)} b(w)\ (z-w)^{-n-1},$$ where $\sim$ means equal modulo the term $:a(z)b(w):$, which is regular at $z=w$. 

Note that $:a(w)b(w):$ is a well-defined element of $\text{QO}(V)$. It is called the {\it Wick product} or {\it normally ordered product} of $a$ and $b$, and it
coincides with $a_{(-1)}b$. For $n\geq 1$ we have
$$ n!\ a(z)_{(-n-1)} b(z)=\ :(\partial^n a(z))b(z):,\qquad \partial = \frac{d}{dz}.$$
For $a_1(z),\dots ,a_k(z)\in \text{QO}(V)$, the $k$-fold iterated Wick product is defined inductively by
\begin{equation}\label{iteratedwick} :a_1(z)a_2(z)\cdots a_k(z):\ =\ :a_1(z)b(z):,\qquad b(z)=\ :a_2(z)\cdots a_k(z):.\end{equation}
We often omit the formal variable $z$ when no confusion can arise.

A subspace $\cA\subseteq \text{QO}(V)$ containing $1$ which is closed under all the above products will be called a {\it quantum operator algebra} (QOA). We say that $a,b\in \text{QO}(V)$ are {\it local} if $$(z-w)^N [a(z),b(w)]=0$$ for some $N\geq 0$. Here $[,]$ denotes the super bracket. This condition implies that $a_{(n)}b = 0$ for $n\geq N$, so (\ref{opeform}) becomes a finite sum. Finally, a {\it vertex algebra} will be a QOA whose elements are pairwise local. This notion is well known to be equivalent to the notion of a vertex algebra in the sense of \cite{FLM}. 

A vertex algebra $\cA$ is said to be {\it generated} by a subset $S=\{\alpha^i|\ i\in I\}$ if $\cA$ is spanned by words in the letters $\alpha^i$, and all products, for $i\in I$ and $n\in\mathbb{Z}$. We say that $S$ {\it strongly generates} $\cA$ if $\cA$ is spanned by words in the letters $\alpha^i$, and all products for $n<0$. Equivalently, $\cA$ is spanned by $$\{ :\partial^{k_1} \alpha^{i_1}\cdots \partial^{k_m} \alpha^{i_m}:| \ i_1,\dots,i_m \in I,\ k_1,\dots,k_m \geq 0\}.$$ 

Suppose that $S$ is an ordered strong generating set $\{\alpha^1, \alpha^2,\dots\}$ for $\cA$ which is at most countable. We say that $S$ {\it freely generates} $\cA$, if $\cA$ has a PBW basis consisting of all normally ordered monomials 
\begin{equation} \label{freegen} \begin{split} & :\partial^{k^1_1} \alpha^{i_1} \cdots \partial^{k^1_{r_1}}\alpha^{i_1} \partial^{k^2_1} \alpha^{i_2} \cdots \partial^{k^2_{r_2}}\alpha^{i_2}
 \cdots \partial^{k^n_1} \alpha^{i_n} \cdots \partial^{k^n_{r_n}} \alpha^{i_n}:,\qquad 
 1\leq i_1 < \dots < i_n,
 \\ & k^1_1\geq k^1_2\geq \cdots \geq k^1_{r_1},\quad k^2_1\geq k^2_2\geq \cdots \geq k^2_{r_2},  \ \ \cdots,\ \  k^n_1\geq k^n_2\geq \cdots \geq k^n_{r_n},
 \\ &  k^{t}_1 > k^t_2 > \dots > k^t_{r_t} \ \ \text{if} \ \ \alpha^{i_t}\ \ \text{is odd}. 
 \end{split} \end{equation}

A conformal structure with central charge $c$ on $\cA$ is a Virasoro vector $L(z) = \sum_{n\in \mathbb{Z}} L_n z^{-n-2} \in \cA$ satisfying
\begin{equation} \label{virope} L(z) L(w) \sim \frac{c}{2}(z-w)^{-4} + 2 L(w)(z-w)^{-2} + \partial L(w)(z-w)^{-1},\end{equation} such that in addition, $L_{-1} \alpha = \partial \alpha$ for all $\alpha \in \cA$, and $L_0$ acts diagonalizably on $\cA$. We say that $\alpha$ has conformal weight $d$ if $L_0(\alpha) = d \alpha$, and we denote the conformal weight $d$ subspace by $\cA[d]$. As a matter of notation, we say that a vertex algebra $\cA$ is of type $$\cW(d_1,d_2,\dots)$$ if it has a minimal strong generating set consisting of one even field in each conformal weight $d_1, d_2, \dots $.

\subsection{Vertex algebras over commutative rings}
Given a finitely generated, unital commutative $\mathbb{C}$-algebra $R$, a vertex algebra over $R$ is an $R$-module $\cA$ with a vertex algebra structure defined as above. A  comprehensive theory of vertex algebras over commutative rings was recently developed by Mason \cite{Ma}, but the main difficulties are not present when $R$ is a $\mathbb{C}$-algebra. First, given an $R$-module $M$, let $\text{QO}_R(M)$ be the set of $R$-module homomorphisms $a: M \ra M((z))$, which can be represented power series $$a(z) = \sum_{n\in \mathbb{Z}} a(n) z^{-n-1} \in \text{End}_R(M)[[z,z^{-1}]].$$ Here $a(n) \in \text{End}_R(M)$ is an $R$-module endomorphism, and for each $v\in M$, $a(n) v = 0$ for $n>\!\!>0$. Then $\text{QO}_R(M)$ is again an $R$-module, and we define the bilinear operations $a_{(n)} b$ as before, which are $R$-module homomorphisms $\text{QO}_R(M) \otimes_R \text{QO}_R(M) \ra \text{QO}_R(M)$. A QOA will be an $R$-module $\cA \subseteq \text{QO}_R(M)$ containing $1$ and closed under the above operations. Locality is defined as before, and a vertex algebra over $R$ is a QOA $\cA\subseteq \text{QO}_R(M)$ whose elements are pairwise local. The notions of generation, strong generation, and free generation are unchanged, as is the notion of conformal structure. If $L  \in \cA$ is a field satisfying \eqref{virope}, such that $L_0$ acts on $\cA$ by $\partial$ and $L_1$ acts diagonalizably, we assume we have an $R$-module decomposition $\cA = \bigoplus_{d\in R} \cA[d]$, where $\cA[d]$ is the $L_0$-eigenspace with eigenvalue $d$. In all our examples, the grading will be by either $\mathbb{Z}_{\geq 0}$ or $\frac{1}{2} \mathbb{Z}_{\geq 0}$. 

Next, we recall the notion of simplicity that was introduced in \cite{L6}. Let $\cA$ be a vertex algebra over $R$ with weight grading 
\begin{equation} \label{eq:gradedrvoa} \cA = \bigoplus_{n\geq 0} \cA[n],\qquad \cA[0] \cong R.\end{equation}
A vertex algebra ideal $\cI \subseteq \cA$ is called {\it graded} if $$\cI = \bigoplus_{n\geq 0} \cI[n],\qquad \cI[n] = \cI \cap \cA[n].$$ 
We say that $\cA$ is {\it simple} if for every proper graded ideal $\cI \subseteq \cA$, $\mathcal{I}[0] \neq \{0\}$. This is different from the usual notion of simplicity, namely, that $\cA$ has no nontrivial proper graded ideals, but it coincides with the usual notion if $R$ is a field. If $I \subseteq R$ is a nontrivial proper ideal, it will generate a nontrivial proper graded vertex algebra ideal $I\cdot \cA$. Then $$\cA^I = \cA / (I\cdot \cA)$$ is a vertex algebra over the ring $R/I$. Even if $\cA$ is simple over $R$, $\cA^I$ need not be simple over $R/I$.

\subsection{Affine vertex algebras} Let $\gg$ be a simple, finite-dimensional Lie algebra with dual Coxeter number $h^{\vee}$. The universal affine vertex algebra $V^k(\gg)$ is freely generated by fields $X^{\xi}$ which are linear in $\xi \in \gg$ and satisfy OPEs
\begin{equation}
X^{\xi}(z) X^{\eta}(w) \sim k ( \xi, \eta) (z-w)^{-2} + X^{[\xi,\eta]}(w)(z-w)^{-1},
\end{equation}
where $(\cdot ,\cdot )$ denotes the normalized Killing form $\frac{1}{2h^{\vee}} \langle \cdot,\cdot \rangle$. For all $k\neq -h^{\vee}$, $V^k(\gg)$ has a conformal vector
\begin{equation} \label{sugawara} L^{\gg}  = \frac{1}{2(k+h^{\vee})} \sum_{i=1}^n :X^{\xi_i} X^{\xi'_i}: \end{equation} of central charge $\displaystyle c = \frac{k\ \text{dim}(\gg)}{k+h^{\vee}}$. Here $\{\xi_i\}$ is a basis of $\gg$, and $\{\xi'_i\}$ is the dual basis with respect to $(\cdot,\cdot)$. As a module over the affine Lie algebra $\hat{\gg} = \gg[t,t^{-1}] \oplus \mathbb{C}$, $V^k(\gg)$ is isomorphic to the vacuum $\hat{\gg}$-module. We denote by $L_k(\gg)$ the (simple) quotient of $V^k(\gg)$ by its maximal proper ideal graded by conformal weight. It is often useful to regard $k$ as a formal variable and regard $V^k(\gg)$ as a vertex algebra over the ring $\mathbb{C}[k]$ localized along the multiplicative set generated by $(k+h^{\vee})$. Then for all $k_0 \neq -h^{\vee}$, the vertex algebra $V^{k_0}(\gg)$ coincides with the specialization of the one-parameter vertex algebra $V^k(\gg)$ at the value $k = k_0$. Alternatively, it is the quotient of $V^k(\gg)$ by the ideal generated by $(k - k_0)$.

\subsection{Coset construction} 
Given a vertex algebra $\cV$ and a subalgebra $\cA \subseteq \cV$, the {\it coset} or {\it commutant} of $\cA$ in $\cV$, denoted by $\text{Com}(\cA,\cV)$, is the subalgebra of elements $v\in\cV$ such that $$[a(z),v(w)] = 0,\qquad\forall a\in\cA.$$ This was introduced by Frenkel and Zhu in \cite{FZ}, generalizing earlier constructions in \cite{GKO,KP}. Equivalently, $v\in \text{Com}(\cA,\cV)$ if and only if $a_{(n)} v = 0$ for all $a\in\cA$ and $n\geq 0$. 
Note that if $\cV$ and $\cA$ have Virasoro elements $L^{\cV}$ and $L^{\cA}$, then $\text{Com}(\cA,\cV)$ has Virasoro element $L = L^{\cV} - L^{\cA}$ as long as $ L^{\cV} \neq L^{\cA}$, and $\cA \otimes \text{Com}(\cA,\cV) \hookrightarrow \cV$ is a conformally embedding.

For all our examples in this paper, $\cA$ will be a homomorphic image of $V^k(\gg)$, and $\cV$ will be some larger vertex algebra whose structure constants depend continuously on the parameter $k$, which we denote by $\cV^k$. If we regard $k$ as a formal variable, this means that for $k_0 \in \mathbb{C}$, $\cV^{k_0}$ is the specialization of the one-parameter vertex algebra $\cV^k$ at $k = k_0$, as above.

In this case, $\cC^k = \text{Com}(\cA,\cV)$ is just the invariant space $(\cV^k)^{\gg[t]}$ under the Lie algebra $\gg[t]$. Cosets of this form were studied in a general setting in \cite{CL} under a few additional hypotheses. First, $\cV^k$ admits a well-defined limit $\cV^{\infty} = \lim_{k\ra \infty} \cV^k$. For example, $V^k(\gg)$ admits such a limit, and $\lim_{k\ra \infty} V^k(\gg)$ is isomorphic to the rank $n$ Heisenberg vertex algebra $\cH(n)$, where $n = \text{dim}(\gg)$. Similarly, any minimal $\cW$-algebra admits such a limit; see Section 4 of \cite{ACKL}. Additionally, we assume

\begin{enumerate}
\item The action of $\gg$ on $\cV^k$ lifts to an action of a connected Lie group $G$ with Lie algebra $\gg$.
\item $\cV^{\infty} = \lim_{k\ra \infty} \cV^k \cong \cH(n) \otimes \tilde{\cV}$ for some vertex algebra $\tilde{\cV}$, where $n = \text{dim}(\gg)$. 
\end{enumerate}

Under these hypotheses, $\cC^k$ has a well-defined limit $\cC^{\infty} = \lim_{k\ra \infty} \cC^k$. Moreover, $G$ acts on $\tilde{\cV}$, and by Theorem 6.10 of \cite{CL}, we have an isomorphism $$\cC^{\infty} = \lim_{k\ra \infty} \cC^k \cong  \tilde{\cV}^G.$$ Many features of $\cC^{\infty}$, such as the weights of a strong generating set, graded character, etc., will also hold for $\cC^k$ for generic values of $k$. This reduces many structural questions about $\cC^k$ to questions about the $G$-invariant vertex algebra $\tilde{\cV}^G$. In earlier works \cite{L1,L2,L3,L4,L5}, the strong finite generation of a large class of vertex algebras of the form $\tilde{\cV}^G$ was established, where $G$ is an arbitrary reductive group. This implies the strong finite generation of $\cC^k$ for generic $k$.

There is a word of caution that we need to mention. We may regard $\cC^k$ as a one-parameter vertex algebra with formal variable $k$. However, the specialization of $\cC^k$ at a value $k = k_0$ can be a proper subalgebra of the honest coset $\text{Com}(V^{k_0}(\gg), \cV^{k_0})$. By Corollary 6.7 of \cite{CL},  if $\gg$ is simple this can only occur when $k_0 + h^{\vee} \in \mathbb{Q}_{\leq 0}$; for all other values of $k_0$, these two vertex algebras agree. In this paper, we shall often use the notation $\cC^k$ without specifying whether $k$ is to be regarded as a formal variable or a complex number.

\subsection{Diagonal coset}
Given a simple Lie algebra $\gg$, one of our main objects of study is the {\it diagonal coset} $C^{k_1, k_2}(\gg)$. Given complex numbers $k_1, k_2$, we have the diagonal embedding
$$V^{k_1+k_2}(\gg) \hookrightarrow V^{k_1}(\gg) \otimes V^{k_2}(\gg),$$ and we define 
\begin{equation} \label{eq:defofck1k2} C^{k_1, k_2}(\gg) := \text{Com}\left(V^{k_1+k_2}(\gg), V^{k_1}(\gg) \otimes V^{k_2}(\gg)\right).\end{equation}
As above, we often regard $k_1, k_2$ as formal variables and regard $C^{k_1, k_2}(\gg)$ as a vertex algebra over the ring $\mathbb{C}[k_1, k_2]$ localized along $(k_1 + h^{\vee})$, $(k_2 + h^{\vee})$ and $(k_1+ k_2 + h^{\vee})$. For notational simplicity, we shall not use a different notation when $k_1, k_2$ are regarded as formal variables, and it should be clear from context when this is the case. Since $V^{k_1+k_2}(\gg)$ is generically simple it follows from \cite[Lem. 2.1]{ACKL} that $C^{k_1, k_2}(\gg)$ is a simple vertex algebra for generic values of $k_1, k_2$; equivalently,  $C^{k_1, k_2}(\gg)$ is simple as a vertex algebra over the above ring. As above, for $\ell_1, \ell_2 \in \mathbb{C}$, as long as $\ell_1 + \ell_2 + h^{\vee} \notin \mathbb{Q}_{\leq 0}$, the specialization of the two-parameter vertex algebra $C^{k_1, k_2}(\gg)$ to the values $k_1 = \ell_1$ and $k_2 = \ell_2$, will coincide with the honest coset $\text{Com}\left(V^{\ell_1+\ell_2}(\gg), V^{\ell_1}(\gg) \otimes V^{\ell_2}(\gg)\right)$.

\subsection{Triality and coset as relative semi-infinite cohomology}
Let $\gg$ be a simple Lie algebra with dual Coxeter number $h^\vee$ and consider $V^{k_1}(\gg) \otimes V^{k_2}(\gg) \otimes V^{k_3}(\gg)$ such that 
$k_1 + k_2 + k_3 = -2h^\vee$. Then the following triality of the diagonal coset is well-known to experts. 
\begin{thm}\label{thm:triality}
For generic values of $k_1$ and $k_2$,  
$C^{k_2, k_3}(\gg) \cong C^{k_1, k_3}(\gg) \cong C^{k_2, k_1}(\gg)$.
\end{thm}

We however are not aware of a proof in the literature and so we want to use this section to quickly explain that it is a direct consequence of \cite{F, FGZ}. First, we recall a second point of view on cosets at generic level. For this let $\gg$ be a simple Lie algebra with basis $\mathcal B$ and dual basis $\mathcal B'$. Consider free fermions $\mathcal F(\gg)$ in two copies of the adjoint representation of $\gg$ with generators $\{ b^x, c^{x'} \,  | \, x\in \mathcal B, \, x'\in \mathcal B' \}$ and operator products
\[
b^x(z) c^{y'}(w) \sim \delta_{x, y} (-w)^{-1}.
\]
Let $M$ be a module for $V^{-2h^\vee}(\gg)$ and let $x(z)$ be the field acting on $M$ and corresponding to $x\in \gg$ . Then we define
\[
d(z) := \sum_{x\in \mathcal B} :x(z) c^{x'}(z): - \frac{1}{2} \sum_{x, y \in \mathcal B} :(:b^{[x, y]}(z) c^{x'}(z):)c^{y'}(z):
\] 
and \cite[Prop. 1.2.]{FGZ} says that the zero-mode $d_0$ of $d(z)$ satisfies $d_0^2=0$. Let's write $d$ for $d_0$. 
One then considers the subalgebra $\widetilde{\mathcal F}(\gg)$ of $\mathcal F(\gg)$ generated by the $b^x$ and $\partial c^{x'}$ and the relative complex
\[
C^{\text{rel}}(\gg, d) =  \left(M\otimes \widetilde{\mathcal F}(\gg)\right)^\gg.
\]
$d$ preserves the relative complex \cite[Prop. 1.4.]{FGZ} and the corresponding cohomology is denoted by $H^{\text{rel},\bullet }_{\infty}(\gg, M)$. 
Let $\rho_{\lambda, k}$ be the irreducible highest-weight representation of $\gg \oplus \mathbb C K$ of highest-weight $\lambda$ where $K$ acts by multiplication with the scalar $k$. Then we set
\[
V^k(\lambda) := \text{Ind}_{\gg[t]\oplus \mathbb C K}^{\widehat{\gg}} \rho_{\lambda, k}
\]
and consider the case $M= V^k(\lambda) \otimes V^{-2h^\vee-k}(\mu)$ for $\lambda, \mu \in P^+$ and for generic $k$. Then the first part of the main Theorem of \cite{F} applies (see \cite[Thm. 1.12]{FGZ} for a proof), i.e.
$H^{\text{rel}, n }_{\infty}(\gg, V^k(\lambda) \otimes V^{-2h^\vee-k}(\mu))=0$ for $n\neq 0$. It follows that the character of $H^{\text{rel}, 0 }_{\infty}(\gg, V^k(\lambda) \otimes V^{-2h^\vee-k}(\mu))$ is the Euler-Poincar\'e character of $C^{\text{rel}}(\gg, d)$ which is obtained from the super character of the involved modules ($\Delta_+$ denotes the set of positive roots)
\begin{equation}
\begin{split}
\Pi(z, q) :&= \prod\limits_{n=1}^\infty (1-q^n)^{\text{rank}(\gg)}\prod\limits_{\alpha \in \Delta_+} (1-z^\alpha q^n)(1-z^{-\alpha}q^n)\\
\text{ch}[V^k(\lambda)](z, q) &= \frac{q^{\frac{(\lambda+\rho)^2}{2(k+h^\vee)}-\frac{\text{dim}(\gg)}{24}}\chi_{\rho_\lambda}(z)}{\Pi(z, q)} \\
\text{sch}[ \widetilde{\mathcal F}(\gg)] &= q^{\frac{\text{dim}(\gg)}{12}} \Pi(z, q)^2
\end{split}
\end{equation}
so that 
\[
\text{sch}[V^k(\lambda) \otimes V^{-2h^\vee-k}(\mu)\otimes \widetilde{\mathcal F}(\gg)] =  q^{\frac{(\lambda+\rho)^2-(\mu+\rho)^2}{2(k+h^\vee)} }\chi_{\rho_\lambda}(z)\chi_{\rho_\mu}(z)=  q^{\frac{(\lambda+\rho)^2-(\mu+\rho)^2}{2(k+h^\vee)} }\chi_{\rho_\lambda \otimes \rho_\mu}(z)
\]
and the $\gg$-invariant part, that is the multiplicity of $\chi_0(z)$ is the Euler-Poincar\'e character of $C^{\text{rel}}(\gg, d)$. This is equal to one if $\mu=-\omega_0(\lambda)$\footnote{$\omega_0$ is the unique Weyl group element that interchanges the fundamental Weyl chamber with its negative. }, i.e. $\rho_\mu$ is the dual of $\rho_\lambda$, and zero otherwise. We thus have
\[
H^{\text{rel}, 0 }_{\infty}(\gg, V^k(\lambda) \otimes V^{-2h^\vee-k}(\mu)) = \begin{cases} \mathbb C & \ \text{if} \ \mu=-\omega_0(\lambda) \\ 0 & \ \text{otherwise} \end{cases}
\]
This provides a proof of the second part of the main Theorem of \cite{F} applied to our setting. 

It is now easy to prove Theorem \ref{thm:triality}:
\begin{proof}
Consider $V^{k_1}(\gg) \otimes V^{k_2}(\gg) \otimes V^{k_3}(\gg)$ such that 
$k_1 + k_2 + k_3 = -2h^\vee$ and define the diagonal coset $C^{k, \ell}(\gg) := \text{Com}\left(V^{k+\ell}(\gg), V^k(\gg) \otimes V^\ell(\gg)\right)$. Let $k_1, k_2$ be generic and decompose $V^{k_2}(\gg) \otimes V^{k_3}(\gg)$ into $V^{k_2+k_3}(\gg) \otimes C^{k_2, k_3}(\gg)$-modules to obtain
\[
H^{\text{rel}, 0 }_{\infty}(\gg, V^{k_1}(\gg) \otimes V^{k_2}(\gg) \otimes V^{k_3}(\gg))  = C^{k_2, k_3}(\gg).
\]
Interchanging the role of $k_1$ and $k_2$ respectively $k_1$ and $k_3$ we similarly obtain 
\[
H^{\text{rel}, 0 }_{\infty}(\gg, V^{k_1}(\gg) \otimes V^{k_2}(\gg) \otimes V^{k_3}(\gg))  = C^{k_2, k_3}(\gg) = C^{k_1, k_3}(\gg) = C^{k_2, k_1}(\gg).
\]
\end{proof}

\subsection{Large and small $N=4$ superconformal vertex algebras}
The large $N=4$ superconformal vertex algebra $V^{k,a}$ arises as the minimal $\cW$-algebra of $\gd(2,1; a)$. It depends on two complex parameters $k$ and $a$, and has strong generators $\{e, f, h, e', f', h', L, G^{\pm\pm}\}$. Here $L$ is a Virasoro field of central charge $c=-6k-3$, $e, f, h, e', f', h'$ are even primary fields of weight $1$, and $G^{\pm\pm}$ are odd primary fields of weight $3/2$. The operator product algebra appears explicitly in \cite{KW}, and is reproduced in Appendix \ref{app:first}. Note that $\{e,f,h\}$ and $\{e',f',h'\}$ generate two commuting affine $\mathfrak{sl}_2$ vertex algebras at levels $\ell = - \frac{a+1}{a}k-1$ and $\ell' = -(a+1)k-1$, respectively. We shall denote by $V_{k,a}$ the simple quotient of $V^{k,a}$ by its maximal proper ideal $\cI_{k,a}$ graded by conformal weight. For generic values of $k$ and $a$, $V^{k,a}$ is simple so that $V^{k,a} = V_{k,a}$

The small $N=4$ algebra $V^k$ is likewise the minimal $\cW$-algebra of $\gp\gs\gl(2|2)$, and has generators $\{e,f,h, L, G^{\pm, \pm}\}$, where $\{e,f,h\}$ form a copy of $V^{-k-1}(\gs\gl_2)$ and $G^{\pm, \pm}$ are odd weight $3/2$ primary fields.

There is a relationship between the two that was elucidated in \cite{CGL}. By rescaling the fields $e', f', h'$ by $1/ a$, and replacing the Virasoro field $L$ in $V^{k,a}$ with the coset Virasoro field $L' = L - L^{\gs\gl_2}$ for $\text{Com}(V^{\ell'}(\gs\gl_2), V^{k,a})$, there is a well-defined limit of $V^{k,a}$ as $a\mapsto \infty$, which we denote by $V^{k,\infty}$. In this limit, $e', f', h'$ become central, and hence generete a commutative Heisenberg $Z$ vertex algebra of rank $3$. We have an exact sequence of vertex algebras 
$$ 0 \rightarrow \cI_{Z}  \rightarrow V^{k,\infty} \rightarrow V^k \rightarrow 0.$$ Here $\cI_Z$ is the ideal generated by $Z$. In particular, $V^k$ is obtained as a quotient of $V^{k,\infty}$.

Consider the coset
\begin{equation} D^{k,a} = \text{Com}(V^{\ell}(\gs\gl_2) \otimes V^{\ell'}(\gs\gl_2), V^{k,a}).\end{equation} A surprising result which we shall prove in the next section, is that we have an isomorphism of two-parameter vertex algebras $C^{k_1, k_2}(\gs\gl_2) \cong D^{k,a}$; see Corollary \ref{cor:cosetiso}. Here the parameters are related by 
$$ k_1 = - \frac{1 + k + a k}{(1 + a) k}, \qquad k_2 = - \frac{a + k + a k}{(1 + a) k}.$$ Note that the symmetry $k_1 \leftrightarrow k_2$ of $C^{k_1, k_2}(\gs\gl_2)$ corresponds to the symmetry $a \leftrightarrow \frac{1}{a}$ of $D^{k,a}$.

Similarly, we will consider the coset 
\begin{equation} D^k = \text{Com}(V^{-k-1}(\gs\gl_2), V^k).\end{equation} The outer automorphism group $\text{SL}_2$ of $V^k$ acts on $D^k$, and we shall prove that $(D^k)^{\text{SL}_2} \cong C^{k_1, -1}$ as one-parametrer vertex algebras; see Corollary \ref{cor:smallN4}. Here the parameters are related by $\displaystyle k = -\frac{1}{1+k_1}$.

\section{Coset construction of large and small $N=4$ superconformal algebras}
It is an important problem to construct $\cW$-algebras as cosets. Recently, coset constructions of the principal $\cW$-algebras for simply-laced $\gg$ were given in full generality \cite{ACL}. Here, we give coset constructions for both the large and small superconformal algebras. These are the first examples of coset realizations of non-principal $\cW$-algebras, in this case minimal $\cW$-algebras, as one-parameter families. We shall see that these realizations are closely related to the diagonal coset $C^{k_1, k_2}(\gs\gl_2)$, which for convenience we denote by $C^{k_1, k_2}$.

\subsection{The large $N=4$ super conformal algebra as a coset}\label{sec:N4coset}

Let $\lambda, \mu$ be generic complex numbers, and define $k_1 = \lambda^{-1} -1, k_2= \lambda -1, \ell_1 = \mu^{-1} -1, \ell_2=\mu -1$. Denote by $V^k(\lambda)$ the irreducible highest-weight module of highest weight $\lambda$  and level $k$ of the affine Lie algebra of $\gs\gl_2$. Let us also write $V^k(n)$ for $V^k(n\omega)$ with $\omega$ the fundamental weight of $\gs\gl_2$. 
Then from \cite{CG} we have that 
\[
L_1(\gd(2, 1;-\lambda)) \cong \bigoplus_{n=0}^\infty V^{k_1}(n) \otimes V^{k_2}(n) \otimes L_1(\bar n)
 \]
with $L_1(\gs\gl_2)$ the simple affine vertex operator algebra of $\gs\gl_2$ at level one and $\bar n$ equal to zero if $n$ is even and one otherwise. We consider
\[
X(\lambda, \mu) := L_1(\gd(2, 1;-\lambda)) \otimes L_1(\gd(2, 1;-\mu))
\]
which carries an action of $\mathbb{Z}_2$ induced from parity on each factor. The decomposition into $\mathbb{Z}_2$-modules is
\begin{equation}
\begin{split}
X(\lambda, \mu) &= \bigoplus_{a, b \in \{0, 1\}} X(\lambda, \mu)_{a, b} \\
X(\lambda, \mu)_{a, b} &= \bigoplus_{\substack{ n, m=0\\  (n, m) = (a, b) \mod 2}}^\infty V^{k_1}(n) \otimes V^{k_2}(n) \otimes L_1(\bar n)\otimes V^{\ell_1}(m) \otimes V^{\ell_2}(m) \otimes L_1(\bar m)
\end{split}
\end{equation}
and especially the even subalgebra decomposes as $ X(\lambda, \mu)_{\text{even}} \cong X(\lambda, \mu)_{0, 0}\oplus  X(\lambda, \mu)_{1, 1}$. We now assume $\lambda, \mu$ to be irrational so that by \cite{KL1, KL2, KL3, KL4, Zha, Hu} the categories of highest-weight modules for $V^k(\gs\gl_2)$ for $k\in \{k_1, k_2, \ell_1, \ell_2\}$ all have rigid vertex tensor category structure so that by \cite{CKM} the $X(\lambda, \mu)_{0, 0}$-modules appearing in $ X(\lambda, \mu)_{\text{even}} $ also belong to a vertex tensor category and so by 
  \cite[Theorem 3. 1]{CKLR} $X(\lambda, \mu)_{1, 1}$ is a simple current for $X(\lambda, \mu)_{0, 0}$, i.e.
\[
X(\lambda, \mu)_{1, 1} \boxtimes_{X(\lambda, \mu)_{0, 0}} X(\lambda, \mu)_{1, 1} \cong X(\lambda, \mu)_{0, 0}.
\]
Let 
\[
Y(\lambda, \mu)_a := \bigoplus_{\substack{ n, m=0\\  n, m = a \mod 2}}^\infty V^{k_1}(n) \otimes V^{k_2}(n) \otimes V^{\ell_1}(m) \otimes V^{\ell_2}(m) 
\]
and recall that the vertex superalgebra of four free fermions $F(4)$ decomposes as a  $L_1(\gs\gl_2) \otimes  L_1(\gs\gl_2)$ as
\[
F(4) \cong L_1(\gs\gl_2) \otimes  L_1(\gs\gl_2) \oplus L_1(1) \otimes  L_1(1)
\]
with even part $L_1(\gs\gl_2) \otimes  L_1(\gs\gl_2)$ and odd one being $ L_1(1) \otimes  L_1(1)$. 
We thus have that 
\[
X(\lambda, \mu)_{\text{even}} \cong Y(\lambda, \mu)_0 \otimes F(4)_{\text{even}} \oplus Y(\lambda, \mu)_1 \otimes F(4)_{\text{odd}}.
\]
By the main Theorem of \cite{CKM2} this implies that there is a braid-reversed equivalence between the vertex tensor subcategory of $F(4)_{\text{even}}$ whose simple inequivalent objects are $F(4)_{\text{even}}$ and $F(4)_{\text{odd}}$ and the category of $Y(\lambda, \mu)_0$ modules with simples $Y(\lambda, \mu)_0$ and $Y(\lambda, \mu)_1$. Since $F(4) \cong F(4)_{\text{even}} \oplus F(4)_{\text{odd}}$ has the structure of a simple vertex operator superalgebra the same must be true for $Y(\lambda, \mu)= Y(\lambda, \mu)_0 \oplus Y(\lambda, \mu)_1$. 
\begin{thm}\label{thm:N4}
For irrational $\lambda, \mu$ we have that 
\[
V^{-\frac{\lambda\mu}{\lambda+\mu}, \frac{\mu}{\lambda}} \cong  \text{Com}\left(V^{k_1 +\ell_1}(\gs\gl_2), Y(\lambda, \mu)\right).
\]
\end{thm}
\begin{proof} Let $C(\lambda, \mu):= \text{Com}\left(V^{k_1+\ell_1}(\gs\gl_2), Y(\lambda, \mu)\right)$ and let
 $k=-\frac{\lambda\mu}{\lambda+\mu}$ and $a= \frac{\mu}{\lambda}$.
\begin{enumerate}
\item
Both $V^{k, a}$ and $C(\lambda, \mu)$ contain $V^{k_2}(\gs\gl_2) \otimes V^{\ell_2}(\gs\gl_2)$ as subalgebra, and moreover we check that the top level subspace of $Y(\lambda, \mu)_1$ has conformal weight $3/2$. Restricting to $C(\lambda, \mu)$ we see that $C(\lambda, \mu)$ has four fields of conformal weight $3/2$ which of course carry the tensor product of the standard representations of the two affine $\gs\gl_2$'s. 
 \item Characters coincide by Corollary \ref{cor:char} that will be proven below.
  \item   $C(\lambda, \mu)$ is a simple vertex algebra by \cite[Lemma 2.1]{ACKL} and by the first step it has a subalgebra that is weakly generated by the weight one and weight $3/2$ fields together with the Virasoro field. By the second step this subalgebra closes under OPE. It thus follows from the uniqueness Theorem \cite[Theorem 3.1]{ACKL} that this subalgebra is isomorphic to $V^{k, a}$ and by the character argument of step two this is everything. 
 \end{enumerate}
\end{proof}
Since $ \text{Com}\left(V^{k_1 +\ell_1}(\gs\gl_2)\otimes V^{k_2}(\gs\gl_2) \otimes V^{\ell_2}(\gs\gl_2), Y(\lambda, \mu)\right) \cong C^{\lambda^{-1}-1, \mu^{-1}-1}$ we immediately have
\begin{cor} \label{cor:cosetiso}
For generic $\lambda, \mu$
\[
D^{-\frac{\lambda\mu}{\lambda+\mu}, \frac{\mu}{\lambda}} \cong C^{\lambda^{-1}-1, \mu^{-1}-1}.
\]
\end{cor}

\subsubsection{Character decomposition}

We need the character decomposition in order to finish Theorem \ref{thm:N4}. It says 
\begin{thm} \label{thm:char}For irrational $\lambda, \mu$
\begin{equation}
\begin{split}
\mathrm{ch}[Y(\lambda, \mu)] &=  \sum_{m=0}^\infty \mathrm{ch}[V^{k_1+\ell_1}](2m)\  V_m
\end{split}
\end{equation}
with 
\[
V_m =  \frac{q^{m+\frac{1}{8}-\frac{(m+1)^2}{4(\lambda^{-1}+\mu^{-1})}}   \prod\limits_{\substack{n=0\\ n\neq m}}^\infty  (1+xq^{n+\frac{1}{2}}) (1+x^{-1}q^{n+\frac{1}{2}}) (1+yq^{n+\frac{1}{2}}) (1+y^{-1}q^{n+\frac{1}{2}})}{\left(\prod\limits_{\substack{n=1\\ n\neq 2m+1}}^\infty (1-q^n) \right) \left(\prod\limits_{n=1}^\infty   (1-q^n)^2 (1-z_1q^n) (1-z_1^{-1}q^n)(1-z_2q^n) (1-z_2^{-1}q^n) \right)}.
\]
\end{thm}
The character of $V^{-\frac{\lambda\mu}{\lambda+\mu}, \frac{\mu}{\lambda}}$ is
\[
\mathrm{ch}[V^{-\frac{\lambda\mu}{\lambda+\mu}, \frac{\mu}{\lambda}}] = q^{-\frac{c}{24}}\frac{\prod\limits_{n=1}^\infty  (1+xq^{n+\frac{1}{2}}) (1+x^{-1}q^{n+\frac{1}{2}}) (1+yq^{n+\frac{1}{2}}) (1+y^{-1}q^{n+\frac{1}{2}})}{\prod\limits_{n=1}^\infty (1-q^{n+1}) (1-q^n)^2  (1-z_1q^n) (1-z_1^{-1}q^n)(1-z_2q^n) (1-z_2^{-1}q^n) }
\]
with $c=-3+\frac{6}{\lambda^{-1}+\mu^{-1}}$.
Especially for $m=0$, Theorem \ref{thm:char} specializes to
\begin{cor}\label{cor:char}
For irrational $\lambda, \mu$
\begin{equation}\nonumber
\begin{split}
V_0 &= \mathrm{ch}\left[\text{Com}\left(V_{k_1 +\ell_1}(\gs\gl_2), Y(\lambda, \mu)\right)\right]  \\
&= q^{\frac{1}{8}-\frac{1}{4(\lambda^{-1}+\mu^{-1})}}  \frac{\prod\limits_{n=1}^\infty  (1+xq^{n+\frac{1}{2}}) (1+x^{-1}q^{n+\frac{1}{2}}) (1+yq^{n+\frac{1}{2}}) (1+y^{-1}q^{n+\frac{1}{2}})}{\prod\limits_{n=1}^\infty (1-q^{n+1})(1-q^n)^2   (1-z_1q^n) (1-z_1^{-1}q^n)(1-z_2q^n) (1-z_2^{-1}q^n) }\\
&= \mathrm{ch}[V^{-\frac{\lambda\mu}{\lambda+\mu}, \frac{\mu}{\lambda}}].
\end{split}
\end{equation}
\end{cor}
We now prove Theorem \ref{thm:char}.
\begin{proof}
We need the theta functions of the root lattice $A_1$
\[
\theta_{A_1}(z) =\theta_{A_1}(z,q) = \sum_{\substack{m\in \mathbb Z\\ m \ \text{even} }} z^m q^{\frac{m^2}{4}}, \qquad \theta_{A_1 + (1)}(z)=\theta_{A_1 + (1)}(z,q) = \sum_{\substack{m\in \mathbb Z\\ m \ \text{odd} }} z^m q^{\frac{m^2}{4}}
\]
and the Weyl denominator of $A_1$
\[
\Pi(z) = \Pi(z, q) = q^{\frac{1}{8}} (z-z^{-1}) \prod_{n=1}^\infty (1-z^2q^n)(1-q^n)(1-z^{-2}q^n).
\]
By \cite[Cor. 9.8]{CG} the character of $L_1(\gd(2, 1;-\lambda)$ for $|z^{\pm 1}q|, |w^{\pm 1}q| < 1$ has the nice form
\[
 \frac{\theta_{A_1}(v)(\theta_{A_1 + (1)}(zw) -\theta_{A_1 + (1)}(zw^{-1}) ) + \theta_{A_1 + (1)}(v) (\theta_{A_1}(zw) -\theta_{A_1}(zw^{-1}) )   }{\eta(q) \Pi(z) \Pi(w)}
\]
with the Dedekind $\eta$-function. From now on $|z_i^{\pm 1}q|, |w_i^{\pm 1}q| < 1$. It follows immediately that 
\begin{equation}
\begin{split}
\text{ch}[Y(\lambda, \mu)] &= \frac{(\theta_{A_1 + (1)}(z_1w_1) -\theta_{A_1 + (1)}(z_1w_1^{-1}) )(\theta_{A_1 + (1)}(z_2w_2) -\theta_{A_1 + (1)}(z_2w_2^{-1}) )  }{\Pi(z_1)\Pi(z_2) \Pi(w_1)\Pi(w_2)} + \\
&\quad  \frac{(\theta_{A_1}(z_1w_1) -\theta_{A_1}(z_1w_1^{-1}) )(\theta_{A_1}(z_2w_2) -\theta_{A_1}(z_2w_2^{-1}) )  }{\Pi(z_1)\Pi(z_2) \Pi(w_1)\Pi(w_2)}.
\end{split}
\end{equation}
Let $F$ be the character of the vertex superalgebra of four free fermions, that is
\begin{equation}
\begin{split}
F(x, y) &= q^{-\frac{1}{12}} \prod_{n=1}^\infty (1+xq^{n-\frac{1}{2}}) (1+x^{-1}q^{n-\frac{1}{2}}) (1+yq^{n-\frac{1}{2}}) (1+y^{-1}q^{n-\frac{1}{2}}) \\
&= \frac{\theta_{A_1}(\sqrt{xy})\theta_{A_1}(\sqrt{xy^{-1}})+\theta_{A_1+(1)}(\sqrt{xy})\theta_{A_1+(1)}(\sqrt{xy^{-1}})}{\eta(q)^2}.
\end{split}
\end{equation}
We set $w=w_1=w_2$,  $x=z_1z_2$ and $y=z_1z_2^{-1}$ so that
\begin{equation}\nonumber
\begin{split}
\text{ch}[Y(\lambda, \mu)] &=  \frac{\left(F(xw^2, y)+ F(x^{-1}w, y)-F(yw^2, x)-F(y^{-1}w, x)\right)}{\eta(q)^{-2} \Pi(z_1)\Pi(z_2) \Pi(w)^2}.
\end{split}
\end{equation}
We now restrict to $|q| < |w| < 1$ and use the magic identity of \cite{KW2}, that is we use equation A.2 of \cite{CR} with $u=-w^2$ and $v=xq^{-\frac{1}{2}}$, to obtain
\[
\frac{\eta(q)^2 F(w^2x, y)}{\Pi(w)} = -\frac{F(x, y)}{\eta(q)} \sum_{m\in \mathbb Z} \frac{w^{2m+1}}{1+xq^{m+\frac{1}{2}}}
\]
and so the coefficients defined by 
\begin{equation}\nonumber
\begin{split}
\Pi(w)\text{ch}[Y(\lambda, \mu)] &=  \sum_{m \in \mathbb Z} w^{2m+1} X_m
\end{split}
\end{equation}
are 
\begin{equation}\nonumber
\begin{split}
X_m &=- \frac{F(x, y)}{\Pi(z_1)\Pi(z_2)\eta(q)} \left( \frac{1}{1+xq^{m+\frac{1}{2}}} +  \frac{1}{1+x^{-1}q^{m+\frac{1}{2}}} - \frac{1}{1+yq^{m+\frac{1}{2}}} -  \frac{1}{1+y^{-1}q^{m+\frac{1}{2}}}\right)\\
&= \frac{F(x, y)}{\Pi(z_1)\Pi(z_2)\eta(q)}  \frac{q^{m+\frac{1}{2}}(1-q^{2m+1})(z_1-z_1^{-1})(z_2-z_2^{-1}) }{(1+xq^{m+\frac{1}{2}})(1+x^{-1}q^{m+\frac{1}{2}}) (1+yq^{m+\frac{1}{2}})(1+y^{-1}q^{m+\frac{1}{2}})}\\
&= q^{m+\frac{1}{8}} \frac{\prod\limits_{\substack{n=0\\ n\neq m}}^\infty  (1+xq^{n+\frac{1}{2}}) (1+x^{-1}q^{n+\frac{1}{2}}) (1+yq^{n+\frac{1}{2}}) (1+y^{-1}q^{n+\frac{1}{2}})}{\left(\prod\limits_{\substack{n=1\\ n\neq 2m+1}}^\infty (1-q^n) \right) \left(\prod\limits_{n=1}^\infty  (1-q^n)^2(1-z_1q^n) (1-z_1^{-1}q^n)(1-z_2q^n) (1-z_2^{-1}q^n) \right)}
\end{split}
\end{equation}
They satisfy $X_m=-X_{-m-1}$ and so it follows that 
\begin{equation}
\begin{split}
\text{ch}[Y(\lambda, \mu)] &= \sum_{m=0}^\infty \frac{(w^{2m+1}-w^{-2m-1})}{\Pi(w)}  X_m\\
&= \sum_{m=0}^\infty \frac{q^{\frac{(m+1)^2}{4(\lambda^{-1}+\mu^{-1})}}(w^{2m+1}-w^{-2m-1})}{\Pi(w)} X_mq^{-\frac{(m+1)^2}{4(\lambda^{-1}+\mu^{-1})}} \\
&= \sum_{m=0}^\infty \text{ch}[V^{k_1+\ell_1}](2m) X_mq^{-\frac{(m+1)^2}{4(\lambda^{-1}+\mu^{-1})}} 
\end{split}
\end{equation}
\end{proof}

\subsection{The small $N=4$ super conformal algebra as a coset}

A slight modification of the above construction gives the small $N=4$ superconformal algebra. 
Let $\lambda$ be a generic complex number, and define $k_1 = \lambda^{-1} -1, k_2= \lambda -1$. Denote by $\rho_n$ the irreducible highest-weight representation of $\gs\gl_2$ of highest weight $n\omega$. Also from \cite{CG} we have that 
\[
L_1(\gp\gs\gl(2|2)) \cong \bigoplus_{n=0}^\infty V^{-1}(n) \otimes \rho_n \otimes L_1(\bar n)
 \]
 Note that $-1$ is a generic level for $\gs\gl_2$ and existence of tensor category follows from \cite{McRae} (see also \cite{CY}).
 We consider
\[
X(\lambda) := L_1(\gd(2, 1;-\lambda)) \otimes L_1(\gp\gs\l(2|2))
\]
which carries an action of $\mathbb Z_2 \times \mathbb Z_2$ induced from parity on each factor. The decomposition into $\mathbb Z_2 \times \mathbb Z_2$-modules is
\begin{equation}
\begin{split}
X(\lambda) &= \bigoplus_{a, b \in \{0, 1\}} X(\lambda)_{a, b} \\
X(\lambda)_{a, b} &= \bigoplus_{\substack{ n, m=0\\  (n, m) = (a, b) \mod 2}}^\infty V^{k_1}(n) \otimes V^{k_2}(n) \otimes L_1(\bar n)\otimes V^{-1}(m) \otimes \rho_m \otimes L_1(\bar m)
\end{split}
\end{equation}
Then as in the last section we define
\[
Y(\lambda)_a := \bigoplus_{\substack{ n, m=0\\  n, m = a \mod 2}}^\infty V^{k_1}(n) \otimes V^{k_2}(n) \otimes V^{-1}(m) \otimes \rho_m 
\]
and $Y(\lambda)= Y(\lambda)_0 \oplus Y(\lambda)_1$ and by the same reasoning as in last subsection $Y(\lambda)$ has the structure of a simple vertex superalgebra. 
The proof of the following Theorems are the same argument as the one of Theorems \ref{thm:N4} and \ref{thm:char}.
\begin{thm}\label{thm:smallN4}
For irrational $\lambda$ we have that 
\[
V^{-\lambda} \cong  \text{Com}\left(V^{k_1 -1}(\gs\gl_2), Y(\lambda)\right).
\]
\end{thm}
\begin{thm} \label{thm:charsmall}For irrational $\lambda$
\begin{equation}
\begin{split}
\mathrm{ch}[Y(\lambda)] &=  \sum_{m=0}^\infty \mathrm{ch}[V^{k_1-1}](2m)\  W_m
\end{split}
\end{equation}
with 
\[
W_m =  \frac{q^{m+\frac{1}{8}-\lambda\frac{(m+1)^2}{4}}   \prod\limits_{\substack{n=0\\ n\neq m}}^\infty  (1+xq^{n+\frac{1}{2}}) (1+x^{-1}q^{n+\frac{1}{2}}) (1+yq^{n+\frac{1}{2}}) (1+y^{-1}q^{n+\frac{1}{2}})}{\left(\prod\limits_{\substack{n=1\\ n\neq 2m+1}}^\infty (1-q^n) \right) \left(\prod\limits_{n=1}^\infty   (1-q^n) (1-z_1q^n) (1-z_1^{-1}q^n)\right)}.
\]
\end{thm}
and hence analogous to Corollary \ref{cor:char}
\begin{cor}
\[
\mathrm{ch}[V^{-\lambda}] =\mathrm{ch}\left[  \text{Com}\left(V^{k_1 -1}(\gs\gl_2), Y(\lambda)\right)\right].
\]
\end{cor}
Since $ \text{Com}\left(V^{k_1 -1}(\gs\gl_2)\otimes V^{k_2}(\gs\gl_2), Y(\lambda, \mu)\right)^{\text{SL}_2} \cong C^{\lambda^{-1}-1, -1}$ we immediately have
\begin{cor}\label{cor:smallN4}
For generic $\lambda$
\[
(D^{-\lambda})^{\text{SL}_2} \cong C^{\lambda^{-1}-1, -1}.
\]
\end{cor}

\section{Generic structure of $C^{k_1, k_2}$}\label{sec:generic}

In this section, we give a rigorous proof of the following statement which has appeared without proof earlier in the physics literature \cite{B-H}, and very recently in \cite{KS}.

\begin{thm} The $\gs\gl_2$ diagonal coset $C^{k_1, k_2}$ is of type $\cW(2,4,6,6,8,8,9,10,10,12)$ as a vertex algebra over a localization of the ring $\mathbb{C}[k_1, k_2]$. Equivalently, this holds for generic values of $k_1$ and $k_2$.
\end{thm}

\begin{proof}
Let $k_2$ be irrational, and regard $k_1$ as a formal variable. Then by Theorem 6.10 of \cite{CL}, we have 
$$\lim_{k_1 \ra \infty} C^{k_1, k_2} \cong V^{k_2}(\gs\gl_2)^{\text{SL}_2}.$$ Moreover, a strong generating set for $V^{k_2}(\gs\gl_2)^{\text{SL}_2}$ will give rise to a strong generating set for $ C^{k_1, k_2}$ for generic values of $k_1$. 

Next, we regard $k_2$ as a formal variable. Then we have 
$$\lim_{k_2 \ra \infty} V^{k_2}(\gs\gl_2)^{\text{SL}_2} \cong \cH(3)^{\text{SL}_2},$$ where $\cH(3)$ is the rank $3$ Heisenberg vertex algebra, and the action of $\text{SL}_2$ on the weight $1$ subspace of $\cH(3)$ is the adjoint representation. Note that  $\cH(3)^{\text{SL}_2} \cong \cH(3)^{\text{SO}_3}$, where action of $\text{SO}_3$ on the the weight one subspace is the standard representation of $\text{SO}_3$.

We need the following case of Weyl's first and second fundamental theorems of invariant theory for the orthogonal group \cite{We}. For $n\geq 0$, let $V_n$ be a copy of the standard representation $\mathbb{C}^3$ of $\gs\go_3$, with orthonormal basis $\{a^1_n, a^2_n, a^3_n\}$. Then $(\text{Sym} \bigoplus_{n=0}^{\infty}V_n)^{\text{SO}_3}$ is generated by 
\begin{equation}\label{quadgen} q_{ij}= a^1_i a^1_j + a^2_i a^2_j +a^3_i a^3_k,\qquad  i,j\geq 0, \end{equation}
\begin{equation}\label{cubgen}c_{klm}= \left| \begin{array}{lll} a^1_k & a^2_k & a^2_k \\ a^1_l & a^2_l & a^3_l \\ a^1_m & a^2_m & a^3_m \end{array}\right|,\qquad  0\leq k<l<m. \end{equation} The ideal of relations among the variables $q_{ij}$ and $c_{klm}$ is generated by polynomials of the following two types: 
\begin{equation}\label{firstrel} q_{ij}c_{klm}-q_{kj}c_{ilm}+q_{lj}c_{kim}-q_{mj}c_{kli}, \end{equation}
\begin{equation}\label{secrel} c_{ijk}c_{lmn}- \left| \begin{array}{lll} q_{il} & q_{im} & q_{in} \\q_{jl} & q_{jm} & q_{jn} \\ q_{kl} & q_{km} & q_{kn}\end{array}\right| . \end{equation}

We have linear isomorphisms \begin{equation} \label{solin} \cH(3)^{\text{SO}_3}\cong \text{gr}(\cH(3))^{\text{SO}_3} \cong (\text{Sym} \bigoplus_{j\geq 0} V_j)^{\text{SO}_3},\end{equation} and isomorphisms of differential graded rings
\begin{equation} \label{sodr} \text{gr}(\cH(3)^{\text{SO}_3})\cong (\text{Sym} \bigoplus_{j\geq 0} V_j)^{\text{SO}_3}.\end{equation} 

The generating set $\{q_{ij},c_{klm}\}$ for $(\text{Sym} \bigoplus_{j\geq 0} V_j)^{\text{SO}_3}$ corresponds to a strong generating set $\{Q_{ij}, C_{klm}\}$ for $\cH(3)^{\text{SO}_3}$, where 
\begin{equation} \begin{split} Q_{ij} & = \ :\partial^i \alpha^1  \partial^j \alpha^1 + :\partial^i \alpha^2  \partial^j \alpha^2: +:\partial^i \alpha^3  \partial^j \alpha^3:,\\
 C_{klm} &  = \ : \partial^k \alpha^1 \partial^l \alpha^2 \partial^m \alpha^3:  - :\partial^k \alpha^1 \partial^m \alpha^2 \partial^l \alpha^3:  -  :\partial^l \alpha^1 \partial^k \alpha^2  \partial^m \alpha^3:  \\ & + : \partial^l \alpha^1 \partial^m \alpha^2 \partial^k \alpha^3:  + :\partial^m \alpha^1  \partial^k \alpha^2  \partial^l \alpha^3:  -  :\partial^m \alpha^1 \partial^l \alpha^2 \partial^k \alpha^3:.\end{split} \end{equation} 
Note that $Q_{ij}$ has weight $i+j+2$ and $C_{klm}$ has weight $k+l+m+3$. As a module over $\cH(3)^{\text{O}_3}$, $\cH(3)^{\text{SO}_3}$ is the direct sum of two irreducible lowest-weight $\cH(3)^{\text{O}_3}$-modules $M_0 \oplus M_1$, where $M_0 \cong \cH(3)^{\text{O}_3}$, which has lowest-weight vector $1$, and $M_1$ which has lowest-weight vector $C_{012}$ and contains all cubics $C_{klm}$. 

It is known that $\cH(3)^{\text{O}_3}$ has a minimal strong generating set $\{Q_{0,2n}|\ n = 0,1,2,\dots, 8\}$, and hence is of type $\cW(2,4,\dots, 18)$ \cite{L3}. Also, it is generated (but not strongly) by the weight $4$ field $Q_{02}$ \cite{L3}. Since $M_1$ is generated as an $\cH(3)^{\text{O}_3}$-module by $C_{012}$, it follows that $\cH(3)^{\text{SO}_3}$ is generated as a vertex algebra by $Q_{02}$ and $C_{012}$. 

In order to find a minimal strong generating set for $\cH(3)^{\text{SO}_3}$, our first observation is that not all the quadratics $\{Q_{0,2n}|\ n = 0,1,2,\dots, 8\}$ are needed. The classical relation 
$$ (c_{012})^2 - (q_{00} q_{11} q_{22} - q_{00} q_{12} q_{12} - q_{01} q_{01} q_{22} - q_{02} q_{11} q_{02} + q_{01} q_{02} q_{12} + q_{02} q_{12} q_{01}) = 0,$$ which is the relation of type \eqref{firstrel} of minimal weight $12$, does not vanish identically, but it admits a quantum correction which appears explicitly in Appendix \ref{appendix:decoup}. Since the coefficient of the weight $12$ field $Q_{0,10}$ in this relation does not vanish, $Q_{0,10}$ can be expressed as a normally ordered polynomial in $\{C_{012}, Q_{00}, Q_{02}, Q_{04}, Q_{06}, Q_{08}\}$ and their derivatives, and hence is not needed. One can now check by computer that the set 
\begin{equation} \label{mingenset} \{C_{01j} |\ j = 2,4,5,6,8\} \cup \{Q_{0,2k}\ k = 0,1,2,3, 4\},\end{equation}
closes under OPE. The reason is that we can find enough decoupling relations which are quantum corrections of the classical relations \eqref{firstrel} to eliminate all other cubics in weights at most $23$. It follows that \eqref{mingenset} strongly generate a vertex subalgebra of $\cH(3)^{\text{SO}_3}$, which must then be all of $\cH(3)^{\text{SO}(3)}$ since \eqref{mingenset} contains the generators $\{Q_{02}, C_{012}\}$ of $\cH(3)^{\text{SO}(3)}$. The fact that \eqref{mingenset} is a minimal strong generating set is clear from the description of the relations \eqref{firstrel} and \eqref{secrel} in Weyl's theorem, since if there were any additional decoupling relations allowing some of these generators to be eliminated, there would be more relations among the classical generators than there are.
 This proves the claim that $\cH(3)^{\text{SO}(3)}$ is of type $\cW(2,4,6,6,8,8,9,10,10,12)$. Finally, this implies that $V^k(\gs\gl_2)^{\text{SL}_2}$ and $C^{k_1, k_2}$ generically have a minimal strong generating in the same weights. \end{proof}

A few remarks are in order. First, it follows from Corollary \ref{cor:cosetiso} that the coset $D^{k,a}$ is also of type $\cW(2,4,6,6,8,8,9,10,10,12)$ for generic values of $k$ and $a$. It is in fact possible to prove this independently via a similar argument involving classical invariant theory, which we sketch below. In addition, the full OPE algebra is determined by a small number of structure constants which can be checked explicitly by computer, so in this way we obtain an independent proof of the isomorphism in Corollary \ref{cor:cosetiso}.

First, recall the generalized free field algebras $\cT$ and $\cG_{\text{odd}}(s)$ introduced in \cite{ACKL}, where $\cT$ is generated by an even field $t$ satisfying 
$$t(z) t(w) \sim (z-w)^{-4},$$ and $\cG_{\text{odd}}(s)$ is generated by odd fields $\phi^i$, $i=1,\dots, s$ satisfying
$$\phi^i(z) \phi^j(w) \sim \delta_{i,j} (z-w)^{-3}.$$

If we rescale $x,y,h,x',y',h',L$ by $\frac{1}{\sqrt{k}}$ and rescale $G^{\pm \pm}$ by $\frac{1}{k}$ it is immediate from the OPE algebra that  $$\lim_{k\ra \infty} V^{k,a} \cong \cH(6) \otimes \cT \otimes \cG_{\text{odd}}(4).$$ Moreover, it follows from the discussion right after Theorem 4.11 of \cite{ACKL} that $$\lim_{k\ra \infty} D^{k,a} \cong \cT \otimes (\cG_{\text{odd}}(4))^{\text{SL}_2 \times \text{SL}_2}.$$ Also, the action of $\text{SL}_2 \times \text{SL}_2$ on $\mathbb{C}^2 \otimes \mathbb{C}^2$ is the same as the action of $\text{SO}_4$ on its standard module $\mathbb{C}^4$. Since the generator of $\cT$ has weight $2$ and corresponds to the Virasoro field, it suffices to prove that $(\cG_{\text{odd}}(4))^{\text{SO}_4}$ is of type $\cW(4,6,6,8,8,9,10,10,12)$. We work in the standard basis $\phi^i$ as above. 

First, we consider the $\text{O}_4$-invariant algebra  $(\cG_{\text{odd}}(4))^{\text{O}_4}$. It is obvious from Weyl's first fundamental theorem of invariant theory \cite{We} that it is strongly generated by the quadratics $$q^{j,k} = \sum_{i=1}^r :(\partial^j \phi^i)( \partial^k \phi^i):,\qquad 0 \leq j < k,$$ which have weight $3+j+k$. However, not all of them are necessary; by Theorem 4.7 of \cite{ACKL} in the case $s=4$, $(\cG_{\text{odd}}(4))^{\text{O}_4}$ has a minimal strong generating set $\{q^{0,2j+1}|\ 0\leq j \leq 7\}$, and hence is if of type $\cW(4,6,8,\dots, 18)$. It is also clear from an OPE computation that $(\cG_{\text{odd}}(4))^{\text{O}_4}$ is generated by the weight $4$ field $q^{0,1}$.

Next, we consider the $\text{SO}_4$-invariants $(\cG_{\text{odd}}(4))^{\text{SO}_4}$. Again by Weyl's first fundamental theorem for $\text{SO}_4$, $(\cG_{\text{odd}}(4))^{\text{SO}_4}$ is strongly generated by the quadratics $q^{i,j}$ together with degree $4$ elements $$w^{ijkl},\qquad 0 \leq i \leq j \leq k \leq l,$$ which are fermionic determinants, that is, determinants without signs. Note that $w^{i,j,k,l}$ has weight $6+i+j+k+l$. Since $(\cG_{\text{odd}}(4))^{\text{O}_4} =  ((\cG_{\text{odd}}(4))^{\text{SO}_4})^{\mathbb{Z}_2}$, it is clear that as a module over $(\cG_{\text{odd}}(4))^{\text{O}_4}$, $(\cG_{\text{odd}}(4))^{\text{SO}_4}$ is the direct sum of $(\cG_{\text{odd}}(4))^{\text{O}_4}$ and the highest-weight $(\cG_{\text{odd}}(4))^{\text{O}_4}$-module generated by the highest-weight vector $w^{0,0,0,0}$, which has weight $6$. Therefore $(\cG_{\text{odd}}(4))^{\text{SO}_4}$ is generated as a vertex algebra by $q^{0,1}$ and $w^{0,0,0,0}$. To find a minimal strong generating set, it suffices to find a set of fields of the form $q^{i,j}$ and $w^{i,j,k,l}$ that contains $q^{0,1}, w^{0,0,0,0}$ and close under OPE, since it must therefore coincide with the subalgebra generated by $q^{0,1}, w^{0,0,0,0}$, which is all of  $(\cG_{\text{odd}}(4))^{\text{SO}_4}$. A computer calculation shows that the following set closes under OPE: $$\{q^{1,0}, q^{3,0}, q^{5,0}, q^{7,0}, w^{0,0,0,0}, w^{0,0,0,2},w^{0,0,0,3},w^{0,0,2,2},w^{0,0,3,3}\},$$ and therefore strongly generates $(\cG_{\text{odd}}(4))^{\text{SO}_4}$. The fact that it is a minimal strong generating is also evident from Weyl's second fundamental theorem of invariant theory, since there are no relations that allow any of these fields to decouple. This shows that $(\cG_{\text{odd}}(4))^{\text{SO}_4}$ is of type $\cW(4,6,6,8,8,9,10,10,12)$. Since $\cT$ has one strong generator in weight $2$ and commutes with $(\cG_{\text{odd}}(4))^{\text{SO}_4}$, this completes the proof.

\section{One-parameter quotients of $C^{k_1, k_2}$}
Recall that $C^{k_1, k_2}$ is simple as a vertex algebra over the ring $\mathbb{C}[k_1, k_2]$. However, there are prime ideals $J \subseteq \mathbb{C}[k_1, k_2]$ for which the quotient $$C^{J,k_1, k_2} = C^{k_1, k_2} / J \cdot C^{k_1, k_2}$$ is not simple as a vertex algebra over $R = \mathbb{C}[k_1, k_2] / J$. Here $J$ is regarded as a subset of the weight zero space $C^{k_1, k_2}[0] \cong \mathbb{C}[k_1, k_2]$, and $J \cdot C^{k_1, k_2}$ denotes the vertex algebra ideal generated by $J$. In fact, these ideals are all of the form $J = (p) \subseteq \mathbb{C}[k_1, k_2]$ where $p$ is an irreducible factor of some Shapovalov determinant. We call the varieties $V(J) \subseteq \mathbb{C}^2$ the {\it truncation curves}, and we denote by $C^{k_1, k_2}_J$ the simple quotient of $C^{J,k_1, k_2}$ by its maximal proper graded ideal. 

Certainly if $k_2$ is constant $r$ and is either a positive integer or an admissible level for $\gs\gl_2$, the ideal $J = (k_2 - r)$ will have the property that $C^{J,k_1, k_2}$ is not simple. By triality, so will the ideal $J = (k_2 + k_1+r+4)$, etc. However, not all truncation curves have this simple form, and it is an interesting problem to classify the truncation curves.

We briefly recall the universal two-parameter algebra $\cW^{\text{ev}}(c,\lambda)$ of type $\cW(2,4,6,\dots)$, which was recently constructed in \cite{KL}. It is defined over the polynomial ring $\mathbb{C}[c,\lambda]$ and is generated by a Virasoro field $L$ with central charge $c$, and a weight $4$ primary field $W^4$, and is strongly generated by the fields $\{L, W^{2i}|\ i \geq 4\}$ where $W^{2i} = W^4_{(1)} W^{2i-2}$ for $i\geq 3$. The idea of the construction is as follows. First, all structure constants in the OPEs of $L(z) W^{2i}(w)$ and $W^{2j}(z) W^{2k}(w)$ for $2i \leq 12$ and $2j+2k \leq 14$, are uniquely determined by imposing appropriate Jacobi identities among these fields. This computation was carried out using the Mathematica package of Thielemans \cite{T}. Next, it was shown inductively that this data uniquely determines {\it all} structure constants in the OPEs $L(z) W^{2i}(w)$ and $W^{2j}(z) W^{2k}(w)$, if a certain subset of Jacobi identities are imposed. This OPE algebra is equivalent to a nonlinear Lie conformal algebra in the language of De Sole and Kac \cite{DSK}, and $\cW^{\mathrm{ev}}(c,\lambda)$ is its universal enveloping vertex algebra.

$\cW^{\mathrm{ev}}(c,\lambda)$ is simple as a vertex algebra over $\mathbb{C}[c,\lambda]$, but there is a certain discrete family of prime ideals $I = (p(c,\lambda)) \subseteq \mathbb{C}[c,\lambda]$ for which the quotient 
$$\cW^{\mathrm{ev, I}}(c,\lambda) = \cW^{\mathrm{ev}}(c,\lambda)/ I \cdot \cW^{\mathrm{ev}}(c,\lambda),$$ is not simple as a vertex algebra over the ring $\mathbb{C}[c,\lambda] / I$. We denote by $\cW^{\mathrm{ev}}_I(c,\lambda)$ the simple quotient of $\cW^{\mathrm{ev, I}}(c,\lambda)$ by its maximal proper graded ideal $\cI$. After a suitable localization, it turns out that all one-parameter vertex algebras of type $\cW(2,4,6,\dots, 2N)$ for some $N$ satisfying some mild hypotheses, can be obtained as quotients of $\cW^{\mathrm{ev}}(c,\lambda)$ in this way. This includes the principal $\cW$-algebras $\cW^k(\gs\gp_{2n}, f_{\text{prin}})$, the orbifolds $\cW^k(\gs\go_{2m}, f_{\text{prin}})^{\mathbb{Z}_2}$, and many others arising as cosets. The generators $p(c,\lambda)$ for such ideals arise as irreducible factors of Shapovalov determinants, and are in bijection with such one-parameter vertex algebras. The corresponding curves $V(I) \subseteq \mathbb{C}^2$ are called {\it truncation curves}, and using the calculations of Hornfeck appearing in \cite{H}, the explicit truncation curves for $\cW^k(\gs\gp_{2n}, f_{\text{prin}})$ and $\cW^k(\gs\go_{2m}, f_{\text{prin}})^{\mathbb{Z}_2}$ were written down in \cite{KL}.

It is also important to consider $\cW^{\mathrm{ev},I}(c,\lambda)$ when $I\subseteq \mathbb{C}[c,\lambda]$ is a {\it maximal} ideal, which has the form $I = (c- c_0, \lambda- \lambda_0)$ for some $c_0, \lambda_0\in \mathbb{C}$. Then $\cW^{\mathrm{ev},I}(c,\lambda)$ and its quotients are vertex algebras over $\mathbb{C}$. Given two maximal ideals $I_0 = (c- c_0, \lambda- \lambda_0)$ and $I_1 = (c - c_1, \lambda - \lambda_1)$, let $\cW_0$ and $\cW_1$ be the simple quotients of $\cW^{\mathrm{ev},I_0}(c,\lambda)$ and $\cW^{\mathrm{ev},I_1}(c,\lambda)$. Theorem 8.1 of \cite{KL} gives a simple criterion for $\cW_0$ and $\cW_1$ to be isomorphic. Aside from a few degenerate cases, we must have $c_0 = c_1$ and $\lambda_0 = \lambda_1$. This implies that aside from the degenerate cases, all other pointwise coincidences among the simple quotients of one-parameter vertex algebras $\cW^{\mathrm{ev},I}(c,\lambda)$ and $\cW^{\mathrm{ev},J}(c,\lambda)$, correspond to intersection points of their truncation curves $V(I)$ and $V(J)$.

\begin{thm} \label{thm:classificationw246} There are exactly three distinct one-parameter vertex algebras of type $\cW(2,4,6)$ that arise as quotients of $\cW^{\text{ev}}(c,\lambda)$. The corresponding ideals $I_i = (p_i)$, $i=1,2,3$, are as follows.

\begin{equation} \begin{split} p_1 & = 3 (-2633664 + 1806268 c + 101736 c^2 + 5275 c^3 + 85 c^4) 
\\ &- 84 (c-1) (2c-1) (46 + 3 c) (22 + 5 c) (444 + 11 c) \lambda
\\ & + 1029 (c-1)^2 (22 + 5 c)^2 (19104 + 1531 c + 29 c^2) \lambda^2, \end{split} \end{equation}

\begin{equation} p_2 =7  \lambda(c-1) (2c-17) (22 + 5 c) + 82 - 47 c - 10 c^2,\end{equation}

\begin{equation} p_3 = 7 \lambda (c-41) (c-1) (22 + 5 c) -14 + 309 c + 5 c^2. \end{equation}

 \end{thm}
 
 \begin{proof}
 This is a straightforward computation in the algebra $\cW^{\text{ev}}(c,\lambda)$ defined in \cite{KL}. One checks that singular vectors in weight $8$ exist for generic values of $c$ only along these curves, together with one more, namely
 $$p = -196 + 172 c - c^2 + 213444 \lambda^2 - 329868 c \lambda^2 + 30429 c^2 \lambda^2 + 74970 c^3 \lambda^2 + 11025 c^4 \lambda^2.$$ However, along this curve, there is a singular vector in weight $6$, and the resulting algebra is of type $\cW(2,4)$. 
\end{proof}
 
 For $I_1 = (p_1)$, the corresponding simple vertex algebra $\cW^{\text{ev}}_I(c,\lambda) = \cW^{\text{ev},I}(c,\lambda) / \cI$ is $\cW^k(\gs\gp_6, f_{\text{prin}})$; see Corollary 5.4 and Equation (A.1) of \cite{KL} in the case $n=3$. As the next theorem shows, the other two can be realized as simple, one-parameter quotients of $C^{k_1, k_2}_J$ for certain prime ideals $J \subseteq \mathbb{C}[k_1, k_2]$. We use the following notation. For $J = (k_2 -r)$ where $r$ is a constant, the simple graded quotient $C^{k_1, k_2}_J$ will be denoted by $C^{k_1}_r$.

\begin{thm} 
\begin{itemize}
\item[]
\item[(a)] In the case $k_2 = 2$, $C^{k_1}_2 \cong \cW^{\text{ev}}_{I_2}(c,\lambda)$. In particular, it is obtained from $\cW^{\text{ev}}(c,\lambda)$ by setting 
\begin{equation} \label{para:k2} 
\begin{split}
c(k_1) &=  \frac{3 k_1 (6 + k_1)}{2 (2 + k_1) (4 + k_1)},\\
 \lambda(k_1) &=-\frac{2 (2 + k_1) (4 + k_1) (-5248 - 4488 k_1 - 352 k_1^2 + 132 k_1^3 + 11 k_1^4)}{7 (-2 + k_1) (8 + k_1) (68 + 42 k_1 + 7 k_1^2) (352 + 354 k_1 + 59 k_1^2)}.
\end{split} 
\end{equation}

\item[(b)] In the case $k_2 = -1/2$, $C^{k_1}_{-1/2} \cong \cW^{\text{ev}}_{I_3}(c,\lambda)$. In particular, it is obtained from $\cW^{\text{ev}}(c,\lambda)$ by setting 
\begin{equation} \label{para:kh} 
\begin{split}
 c(k_1) &= -\frac{k_1 (7 + 2 k_1)}{(2 + k_1) (3 + 2 k_1)},\\
  \lambda(k_1) &= \frac{(2 + k_1) (3 + 2 k_1) (84 + 2359 k_1 + 3271 k_1^2 + 1484 k_1^3 + 212 k_1^4)}{14 (3 + k_1) (1 + 2 k_1) (41 + 49 k_1 + 14 k_1^2) (132 + 119 k_1 + 34 k_1^2)}.
  \end{split}
  \end{equation}
\end{itemize}
\end{thm}

\begin{proof}
We begin with the case $k_2 = 2$. By rescaling the generators of $V^{k_1}(\gs\gl_2)$ by $\frac{1}{\sqrt{k_1}}$, there is a well-defined limit 
\begin{equation} \label{limit:ck2} \lim_{k_1 \ra \infty} C^{k_1}_2 \cong L_{2}(\gs\gl_2)^{\text{SL}_2}.\end{equation} By Theorem 6.10 of \cite{CL}, $C^{k_1}_2$ has the same strong generating type as $L_{2}(\gs\gl_2)^{\text{SL}_2}$ for generic values of $k_1$. Next, let $F(3)$ denote the algebra of three free fermions, which has automorphism group the orthogonal group $\text{O}_3$. It is an extension of $L_2(\gs\gl_2)$, in fact $F(3) = L_2(\gs\gl_2) \oplus L_2(2 \omega)$. It follows that $L_2(\gs\gl_2)^{\text{SL}_2} \cong F(3)^{\text{O}_3}$ which is of type $\cW(2,4,6)$ by Theorems 6 and 9 of \cite{L5}. Then $C^{k_1}_2$ is also of type $\cW(2,4,6)$ as a one-parameter vertex algebra. It is easily checked that it satisfies the criteria to be a quotient of $\cW^{\text{ev}}(c,\lambda)$, and by computing the fourth order pole of the primary weight $4$ field $W^4$ with itself, the explicit truncation curve can be deduced.

Similarly, for $k_2 = -\frac{1}{2}$, we have 
\begin{equation} \label{limit:ck-1/2} \lim_{k_1 \ra \infty} C^{k_1}_{-1/2} \cong L_{-1/2}(\gs\gl_2)^{\text{SL}_2} \cong \cS^{\text{SL}_2},\end{equation} where $\cS$ denotes the rank one $\beta\gamma$-system. This is known to be of type $\cW(2,4,6)$ \cite{BFH}; see also \cite{L5} for a uniform description of $\cS(n)^{\text{Sp}_{2n}}$. As above, it follows that $C^{k_1}_{-1/2}$ is of type $\cW(2,4,6)$ as a one-parameter vertex algebra, and arises as a quotient of $\cW^{\text{ev}}(c,\lambda)$. It is straightforward to compute the fourth order pole of $W^4$ with itself to deduce the explicit truncation curve. \end{proof}

\begin{remark} It was stated without proof in the physics literature \cite{B-H} that $C^{k_1}_{2}$ and $C^{k_1}_{-1/2}$ are of type $\cW(2,4,6)$. Additionally, it was stated that $C^{k_1}_2$ is isomorphic as a one-parameter vertex algebra to the $\mathbb{Z}_2$-orbifold of the $N=1$ superconformal algebra. This follows from  \cite[Thm 2.10 and Lem. 2.11]{CGL}, and in fact can be reproven in a different way. One checks easily by passing to a suitable limit in which the $N=1$ algebra becomes a generalized free field algebra, that this orbifold of type $\cW(2,4,6)$. By computing the fourth order pole of the weight $4$ primary field $W^4$ with itself, we can find the value of $\lambda$ that realizes it as quotient of $\cW^{\text{ev}}(c,\lambda)$. This manifestly agrees with the above value. \end{remark}

Next, we consider the case $k_1 = k_2$, that is the ideal $J = (k_1 - k_2)$.

\begin{thm} Let $J = (k_1 - k_2)$. Then the simple one-parameter vertex algebra $C^{k_1, k_1}_J$ is isomorphic to $\cW^{\text{ev}}_{I_4}(c,\lambda)$ where the ideal $I_4 = (p_4)$, and
\begin{equation}\label{kkcurve} \begin{split} p_4 &= f(c) + \lambda g(c) + \lambda^2 h(c), \\  f(c) &= -196 + 1476 c - 955 c^2 - 25 c^3, \\ g(c) &= 980 (c-1) (2c-1) (22 + 5 c),\\  h(c) &= 49 (c-25) (c-1)^2 (22 + 5 c)^2.\end{split} \end{equation}
In particular, $C^{k_1, k_1}_J$ is obtained from $\cW^{\text{ev}}(c,\lambda)$ by setting 
\begin{equation} \label{para:kk} c(k_1) = \frac{3 k_1^2}{(1 + k_1) (2 + k_1)},\quad \lambda(k_1) =\frac{(1 + k_1) (2 + k_1) (-28 - 118 k_1 - 23 k_1^2 + 22 k_1^3)}{7 (-2 + k_1) (1 + 2 k_1) (5 + 2 k_1) (44 + 66 k_1 + 37 k_1^2)}.\end{equation}
\end{thm}

\begin{proof}

By Theorem 13.3 of \cite{ACL} (see also \cite[Lemma 4.12]{JL}), for $m\geq 2$ a positive integer, we have an isomorphism of vertex algebras
\begin{equation} \label{eq:aclthm13.3}\cW_{\ell} (\gs\go_{2m}, f_{\text{prin}}) \cong \bigg( \big(L_{2m}(\gs\go_4) \oplus \mathbb{L}_{2m}(2m\ \omega)\big)^{\gs\go_3[t]}\bigg)^{\mathbb{Z}_2},\qquad  \ell = -(2m-2) + \frac{2m+1}{2m+2}.\end{equation} Here $\omega=\omega_1+\omega_2$ and $\omega_1, \omega_2$ denote the  fundamental weights of $\gs\go_4\cong \gs\gl_2 \oplus \gs\gl_2$. 
The conformal weight of the top level of $\mathbb{L}_{2m}(2m\ \omega)$ is $m$. 
We may identify $\gs\go_3$ with $\gs\gl_2$ and $\gs\go_4$ with $\gs\gl_2 \oplus \gs\gl_2$. Also, the action of $\gs\go_3$ on $\gs\go_4$ corresponds to the diagonal action of $\gs\gl_2$ on $\gs\gl_2 \oplus \gs\gl_2$. Therefore the right hand side of \eqref{eq:aclthm13.3} is is a nontrivial extension of the coset $$C_{2m, 2m} = \text{Com}(L_{4m}(\gs\gl_2), L_{2m}(\gs\gl_2) \otimes L_{2m}(\gs\gl_2)).$$

Recall next that $\cW_{\ell}(\gs\go_{2m}, f_{\text{prin}})$ has an action of $\mathbb{Z}_2$, and as a module over the orbifold $\cW_{\ell}(\gs\go_{2m}, f_{\text{prin}})^{\mathbb{Z}_2}$ it decomposes as $\cW_{\ell}(\gs\go_{2m}, f_{\text{prin}})^{\mathbb{Z}_2} \oplus M_m$ where $M_m$ is an irreducible, highest-weight module of conformal weight $m$. Also, by Theorem 9.4 of \cite{KL}, we have 
$$\cW_{\ell}(\gs\go_{2m}, f_{\text{prin}})^{\mathbb{Z}_2} \cong \cW_{\ell'}(\gs\gp_{4m}, f_{\text{prin}}),$$ where $\displaystyle \ell' = -(2m+1)+ \frac{m+1}{2m+1}$. Combining this isomorphism with Remark 5.3 of \cite{KL}, it follows that $\cW_{\ell}(\gs\go_{2m}, f_{\text{prin}})^{\mathbb{Z}_2}$ is generated as a vertex algebra by the weight $4$ field $W^4$, which can be taken to be the unique primary field. Since $\mathbb{L}_{2m}(2m\ \omega_1)$ has conformal weight $m$, it follows that for $m > 4$, the field $W^4$ on the left-hand side of \eqref{eq:aclthm13.3} lies in $C_{2m, 2m}$, so $\cW(\gs\go_{2m}, f_{\text{prin}})^{\mathbb{Z}_2} \subseteq C_{2m, 2m}$. 

Since the right hand side of \eqref{eq:aclthm13.3} also has the decomposition $\cW(\gs\go_{2m}, f_{\text{prin}})^{\mathbb{Z}_2} \oplus M_m$ as  $\cW(\gs\go_{2m}, f_{\text{prin}})^{\mathbb{Z}_2}$-modules, and $C_{2m, 2m}$ is a proper subalgebra, it follows that $$C_{2m, 2m} \cong \cW(\gs\go_{2m}, f_{\text{prin}})^{\mathbb{Z}_2}.$$ for $m  >4$.

By Corollary 8.2 of \cite{CL}, $C_{2m, 2m}$ is the simple quotient of the universal algebra $C^{2m, 2m}$. Since it is a $\cW^{\text{ev}}(c,\lambda)$ quotient for infinitely many values of $m$, it follows that the one-parameter vertex algebra $C^{k_1,k_1}_J$ must also be such a quotient. Now that this is established, the truncation curve given by \eqref{kkcurve} can be found by computing the coefficient of $W^4$ in the fourth-order pole of $W^4$ with itself. \end{proof}

Finally, we consider the case $k_2 = -1$. Since this is a generic level for $\gs\gl_2$, $C^{k_1,-1}$ is simple as a one-parameter vertex algebra, hence $C^{k_1,-1} \cong C^{k_1}_{-1}$. By Corollary \ref{cor:smallN4}, we have 
$$C^{k_1,-1} \cong (D^k)^{\text{SL}_2},\qquad k = -\frac{1}{1+k_1},$$ where $D^k =  \text{Com}(V^{-k-1}(\gs\gl_2), V^k)$. Here $\text{SL}_2$ denotes the group of outer automorphisms of $V^k$, which acts nontrivially on $D^k$. By Theorem 7.3 of \cite{CLR}, under the subgroup $U(1) \subseteq \text{SL}_2$, the orbifold $(D^k)^{U(1)}$ is of type $\cW(2,3,4,5,6,7,8)$, and arises as a one-parameter quotient of the two-parameter $\cW_{\infty}$-algebra $\cW(c,\lambda)$ of type $\cW(2,3,\dots)$, which was constructed in \cite{L6}. Recall that $\cW(c,\lambda)$ is a simple vertex algebra over the ring $\mathbb{C}[c,\lambda]$, and is generated by a Virasoro field and a weight $3$ primary field. Any simple, one-parameter vertex algebra of type $\cW(2,3,\dots, N)$ for some $N$ satisfying mild hypotheses, arises as a quotient of $\cW(c,\lambda)$ of the form $\cW_I(c,\lambda)$. Here $I \subseteq \mathbb{C}[c,\lambda]$ is an ideal such that $\cW^I(c,\lambda)=\cW(c,\lambda) / I\cdot \cW(c,\lambda)$ is not simple as a vertex algebra over $\mathbb{C}[c,\lambda] / I$, and $\cW_I(c,\lambda)$ denotes its unique simple graded quotient. In this notation, $(D^k)^{U(1)} \cong \cW_I(c,\lambda)$ for $I = (\lambda + \frac{1}{16})$. In particular, $(D^k)^{U(1)}$ is obtained from $\cW(c,\lambda)$ by setting
$$ c= \frac{3 k (3 + 2 k)}{2 + k},\qquad \lambda = -\frac{1}{16},$$ and taking the simple quotient.

Clearly $(D^k)^{\text{SL}_2} \subseteq (D^k)^{U(1)}$. Moreover, $(D^k)^{U(1)}$ has full automorphism group $\mathbb{Z}_2$; the proof is the same as the proof of Corollary 5.13 of \cite{L6}, and the action of $\mathbb{Z}_2$ extends to $D^k$. It follows that $(D^k)^{\text{SL}_2} \cong ((D^k)^{U(1)})^{\mathbb{Z}_2}$. We therefore obtain
\begin{thm} \label{thm:onepara-1} As a one-parameter vertex algebra, $C^{k_1,-1} \cong \cW_I(c,\lambda)^{\mathbb{Z}_2}$, where $I = (\lambda + \frac{1}{16})$.  
\end{thm}

\section{Simple quotient $C_{k_1, k_2}$ and coincidences.}

For $r,s \in \mathbb{C}$, let $C_{r,s}$ denote the simple graded quotient of $C^{k_1, k_2}$ along the maximal ideal $(k_1-r,\ k_2 - s)$. It is a simple vertex algebra over $\mathbb{C}$ which is the simple quotient of the coset except possibly if $r+s \in \mathbb{Q}_{\leq -2}$.

It need not always be the case that $L_{k_1+k_2}(\gs\gl_2)$ embeds in $L_{k_1}(\gs\gl_2) \otimes L_{k_2}(\gs\gl_2)$. However, if $k_1$ is a positive integer and $k_2$ is admissible (or vice versa), this is the case, and we have 
\begin{thm} If $k_1$ is a positive integer, and $k_2$ is a positive integer or an admissible level for $\gs\gl_2$ (or vice versa), then 
$$C_{k_1,k_2} = \text{Com}(L_{k_1+k_2}(\gs\gl_2), L_{k_1}(\gs\gl_2) \otimes L_{k_2}(\gs\gl_2)).$$
\end{thm}
\begin{proof}
Let $k_2=n$ be a positive integer and $k_1$ admissible, Then the statement follows by induction for $n$. The case $n=1$ is well-known and actually $C_{k_1, 1}$ has been fully understood in \cite{ACL}. Assume the statement to be true for all admissible levels and all $0\leq m <n$. 
For general $n$ we consider $L_{k_1}(\gs\gl_2) \otimes L_{n-1}(\gs\gl_2) \otimes L_{1}(\gs\gl_2)$. By induction hypothesis we have the chain of embeddings
\[
L_{k_1+n}(\gs\gl_2) \subseteq L_{k_1+n-1}(\gs\gl_2)\otimes L_{1}(\gs\gl_2)\subseteq L_{k_1}(\gs\gl_2) \otimes L_{n-1}(\gs\gl_2) \otimes L_{1}(\gs\gl_2)
\] 
and of course also
\[
 L_{k_1}(\gs\gl_2) \otimes L_{n}(\gs\gl_2)  \subseteq  L_{k_1}(\gs\gl_2) \otimes L_{n-1}(\gs\gl_2) \otimes L_{1}(\gs\gl_2).
\]
The coset by $ \text{Com}\left(L_{k_1}(\gs\gl_2) \otimes L_n(\gs\gl_2), L_{k_1}(\gs\gl_2) \otimes L_{n-1}(\gs\gl_2) \otimes L_{1}(\gs\gl_2)\right)$  is the Virasoro algebra $\text{Vir}_\ell$ at level $\displaystyle \ell = -2 + \frac{n}{n-1}$ and these form a mutually commuting pair by \cite{ACL}. Conversely the Virasoro algebra coset  $\text{Com}\left(\text{Vir}_\ell, L_{k_1}(\gs\gl_2) \otimes L_{n-1}(\gs\gl_2) \otimes L_{1}(\gs\gl_2)\right)$ of course contains $L_{k_1+n}(\gs\gl_2)$ but it is also equal to $L_{k_1}(\gs\gl_2) \otimes L_{n}(\gs\gl_2)$. Hence the claim.
 \end{proof} 
  
 We make the following conjecture. 
 
   \begin{conj} \label{conj:rationality} 
  Let $n$ be a positive integer, and let $k$ be a positive integer or an admissible level for $\gs\gl_2$. Then $C_{k, n}$ is strongly rational.
  \end{conj}
  
  In the case $n = 1$, this is well known since $C_{k,1}$ is then a Virasoro minimal model. We are unable to prove this conjecture in general, but the next result, which has as a corollary the existence of new rational $\cW$-algebras of type $C$, shows that it holds for $n = 2$.
 
\begin{thm} \label{thm:strongrational2}
Let $k$ be a positive integer or an admissible level for $\gs\gl_2$. Then $C_{k, 2}$ is strongly rational.
\end{thm}
\begin{proof}
Let $F(n)$ be the vertex superalgebra of $n$ free fermions
Recall that four free fermions are isomorphic to $F(4) \cong L_{1}(\gs\gl_2) \otimes  L_{1}(\gs\gl_2) \oplus L_{1}(\omega) \otimes  L_{1}(\omega)$ but also
\[
F(3) \cong L_{2}(\gs\gl_2) \oplus L_2(2\omega) 
\]
so that $\text{Com}(L_{k+2}(\gs\gl_2), L_{k}(\gs\gl_2) \otimes F(4) )$ is both an extension of $C_{k, 1} \otimes C_{k+1, 1}$ and a simple current extension of
 $C_{k, 2} \otimes F(1)$. It is strongly rational by \cite[Cor. 1.1]{CKM2} and so as an order two orbifold $C_{k, 2} \otimes F(1)$ is strongly rational as well \cite{CM}. This can only be true if $C_{k, 2}$ is strongly rational.
\end{proof}

\begin{remark}
If one knows that rigid vertex tensor category structure exists on a given vertex algebra $V$, as well as on their subalgebras $W$ and $C=\text{Com}(W, V)$ then it is work in progress that one can determine tensor category structure of $C$-modules in terms of those of $W$ and $V$ \cite{CKM3}. For the case of $C=C_{k, 2}$ the outcome will be: 

The category of ordinary modules of admissible level affine vertex algebras of simply laced Lie algebras is a rigid vertex tensor category \cite{CHY, C}. 
Let $k=-2+\frac{u}{v}$ be an admissible level, that is $u, v$ coprime positive integers and $u\geq 2$. Let $\mathcal O_{u, v}^{\text{ord}}$ be the category of ordinary modules of $L_k(\gs\gl_2)$ and  $\overline{\mathcal O}_{u, v}^{\text{ord}}$ its braid-reversed category. Then the category of modules of $C_{k, 2}$ is equivalent to a category of local modules of an order two simple current extension in $\mathcal O_{u, v}^{\text{ord}}\boxtimes \mathcal O_{4, 1}^{\text{ord}} \boxtimes \overline{{\mathcal O}}_{u+2v, v}^{\text{ord}}$. The simple objects are labelled as $M_{a, b, c}$ with $a =0, \dots, u-2, b=0, 1, 2, c= 0, \dots, u+2v-2$ and if $v$ is even we need that $b$ has to be even as well, while for $v$ odd $a+b+c$ has to be even. 
Moreover the only isomorphisms between simple objects are $M_{a, b, c} \cong M_{u-2-a, 2-b, u+2v-2-c}$.
Finally, the fusion rules are 
\[
M_{a, b, c} \otimes M_{a', b', c'} \cong \bigoplus_{a'', b'', c''} N^{u}_{a, a', a''} N^{4}_{b, b', b''} N^{u+2v}_{c, c', c''} M_{a'', b'', c''}
\]
where the $N^w_{x, x', x''}$ are the fusion rules of $L_{w-2}(\gs\gl_2)$. 
\end{remark}

\subsection{Coincidences}
We are interested in finding coincidences between the simple algebras $C_{k_1, k_2}$ and other vertex algebras such as $\cW_k(\gs\gp_{2n}, f_{\text{prin}})$, $\cW_k(\gs\go_{2n}, f_{\text{prin}})^{\mathbb{Z}_2}$, etc., as well as coincidences $C_{k_1, k_2} \cong C_{\ell_1, \ell_2}$ which are not due to triality.

In the cases $C_{k_1,2}$, $C_{k_1, -1/2}$, and $C_{k_1,k_1}$, we can find certain coincidences using the fact that these algebras arise as quotients of $\cW^{\text{ev}}(c,\lambda)$. Here we use the fact that these coincidences correspond to the intersection points of the truncations curves.

\begin{thm} \label{thm:k2spcoin} For $n\geq 2$, aside from the critical levels $k_1 = -2$, $k_1+2 = -2$ and $\ell = -(n+1)$, and the degenerate cases given by Theorem 8.1 of \cite{KL}, all isomorphisms
$$C_{k_1, 2} \cong \cW_{\ell}(\gs\gp_{2n}, f_{\text{prin}}),$$ appear on the following list.
\begin{enumerate}

\item $\displaystyle k_1 = -\frac{4 n}{1 + 2 n},\ k_1 =  -\frac{2 (3 + 4 n)}{1 + 2 n},\qquad \ell = -(n+1) + \frac{1 + 2 n}{4 (1 + n)}$,

\smallskip

\item $\displaystyle k_1 = \frac{3 - 2 n}{n},\ k_1 = - \frac{3 + 4 n}{n},\qquad \ell = -(n+1) + \frac{3 + 2 n}{4 n}$,

\smallskip

\item $\displaystyle k_1 = -4n,\ k_1 = 4n - 6,\qquad  \ell = -(n+1) + \frac{2n-1}{4 (n-1)}$.
\end{enumerate} \end{thm}

\begin{proof} First, we exclude the values of $k_1$ and $\ell$ which are poles of the functions $\lambda(k)$ given by \eqref{para:k2}, and $\lambda_3(\ell)$ given by Equation (A.2) of \cite{KL} with $k$ replaced by $\ell$, since at these values, $C^{k_1}_2$ and $\cW_{\ell}(\gs\gp_{2n}, f_{\text{prin}})$ are not quotients of $\cW^{\mathrm{ev}}(c,\lambda)$. For all other noncritical values of $k_1$ and $\ell$, $C^{k_1}_2$ and $\cW_{\ell}(\gs\gp_{2n}, f_{\text{prin}})$ are obtained as quotients of $\cW^{\mathrm{ev},J_{n}}(c,\lambda)$ and $\cW^{\mathrm{ev}, J_{m}}(c,\lambda)$, respectively. 

By Corollary 8.2 of \cite{KL}, aside from the degenerate cases given by Theorem 8.1 of \cite{KL}, all other coincidences  $C^{k_1}_2 \cong \cW_{\ell}(\gs\gp_{2n}, f_{\text{prin}})$ correspond to intersection points on the truncation curves $V(J_{2})$ and $V(I_{3})$. Here the generator of $J_2$ is given by Theorem \ref{thm:classificationw246}, and the generator of $I_3$ is given by Equation (A.1) of \cite{KL}. A calculation shows that $V(J_{2}) \cap V(I_{3})$ consists of exactly five points $(c,\lambda)$, namely, 

\begin{equation} \begin{split} &\bigg(-24, -\frac{1}{245}\bigg),\qquad \bigg(\frac{1}{2}, -\frac{2}{49}\bigg),
\\ & \bigg( -\frac{3 n (3 + 4 n)}{2 (1 + n)},\  -\frac{2 (1 + n) (-164 - 751 n - 746 n^2 + 516 n^3 + 720 n^4)}{7 (2 + 3 n) (1 + 4 n) (17 + 26 n + 12 n^2) (-44 + n + 60 n^2)}   \bigg),
\\ & \bigg( -\frac{(2n-3) (3 + 4 n)}{2 (3 + 2 n)}\ -\frac{2 (3 + 2 n) (99 + 132 n - 880 n^2 - 616 n^3 + 160 n^4)}{7 (1 + 2 n) (-3 + 4 n) (21 + 14 n + 4 n^2) (-177 - 118 n + 40 n^2)} \bigg),
\\ & \bigg( \frac{3 n (2n-3)}{2 (n-1) (2n-1)}, \ -\frac{2 (n-1) (2n-1) (-164 + 561 n - 176 n^2 - 264 n^3 + 88 n^4)}{7 (n-2) (1 + 2 n) (17 - 42 n + 28 n^2) (44 - 177 n + 118 n^2)}   \bigg).\end{split} \end{equation}

By Theorem 8.1 of \cite{KL}, the first two intersection points occur at degenerate values of $c$. By replacing the parameter $c$ with the levels $k_1$ ad $\ell$, we see that the remaining intersection points yield the nontrivial isomorphisms in Theorem \ref{thm:k2spcoin}. Moreover, by Corollary 8.2 of \cite{KL}, these are the only such isomorphisms except possibly at the values of $k, \ell$ excluded above. 

Finally, suppose that $k_1$ is a pole of the function $\lambda(k_1)$ given by \eqref{para:k2}. It is not difficult to check that the corresponding values of $\ell$ for which $c(k_1) = c_3(\ell)$, are not poles of $\lambda_3(\ell)$. Here $c(k_1)$ and $\lambda(k_1)$ are given by \eqref{para:k2}, and $c_3(\ell)$ and $\lambda_3(\ell)$ are given by (A.2) of \cite{KL}, with $k$ replaced by $\ell$.  It follows that there are no additional coincidences at the excluded points.
\end{proof}

Note that in examples (1) and (2) above,  $\ell$ is a nondegenerate admissible level for $\gs\gp_{2n}$, but in example (3), $\ell$ is a degenerate admissible level. Since $C_{4n-6, 2}$ is strongly rational by Theorem \ref{thm:strongrational2}, we obtain

\begin{cor} For $n\geq 2$, at the degenerate admissible level $\displaystyle \ell = -(n+1) + \frac{2n-1}{4 (n-1)}$, $\cW_{\ell}(\gs\gp_{2n}, f_{\text{prin}})$ is strongly rational. These are new examples of strongly rational principal $\cW$-algebras.
\end{cor}

The next result classifies the coincidences between $C_{k_1, 2}$ and the simple $D$-type orbifolds $\cW_{\ell}(\gs\go_{2n}, f_{\text{prin}})^{\mathbb{Z}_2}$. The proof is similar to the proof of Theorem \ref{thm:k2spcoin}, and is omitted.
\begin{thm} For $n\geq 3$, aside from the critical levels $k_1 = -2$, $k_1+2 = -2$, and $\ell = -(2n-2)$, and the degenerate cases given by Theorem 8.1 of \cite{KL}, all isomorphisms
$$C_{k_1, 2} \cong \cW_{\ell}(\gs\go_{2n}, f_{\text{prin}})^{\mathbb{Z}_2},$$ appear on the following list.

$\displaystyle k_1 = -\frac{8 n}{2n-1},\ k_1  = -\frac{2 (2n-3)}{2n-1}, \qquad  \ell = -(2n-2) + \frac{2n-1}{2n+1},\  \ell = -(2n-2) + \frac{2n+1}{2n-1}$.
\end{thm}

Next, we classify all coincidences between $C_{k_1, -1/2}$ and the algebras $\cW_{\ell}(\gs\gp_{2n}, f_{\text{prin}})$ and $\cW_{\ell}(\gs\go_{2n}, f_{\text{prin}})^{\mathbb{Z}_2}$.
\begin{thm} \label{thm:c-1/2coinc} For $n\geq 2$, aside from the critical levels $k_1 = -2$, $k_1 - 1/2 = -2$ and $\ell = -(n+1)$, and the degenerate cases given by Theorem 8.1 of \cite{KL}, all isomorphisms
$$C_{k_1,-1/2} \cong \cW_{\ell}(\gs\gp_{2n}, f_{\text{prin}}),$$ appear on the following list.
\begin{enumerate}

\item $\displaystyle k_1 = \frac{1 - 4 n}{2 n},\ k_1 = - \frac{1 + 3 n}{2 n},\qquad \ell = -(n+1) + \frac{-1 + n}{2 n}$,

\smallskip

\item $\displaystyle k_1 = n, \ k_1 = - \frac{1}{2} (7 + 2 n),\qquad  \ell = -(n+1) + \frac{2 + n}{3 + 2 n}$,

\smallskip

\item $\displaystyle k_1 = -\frac{4 n}{1 + 2 n},\ k_1 = -\frac{7 + 6 n}{2 (1 + 2 n)},\qquad \ell = -(n+1) +\frac{1 + 2 n}{2 (2n-3)}$.

\end{enumerate}
\end{thm}

\begin{thm} \label{thm:d-1/2coinc} For $n\geq 3$, aside from the critical levels  $k_1 = -2$, $k_1 - 1/2 = -2$, and $\ell = -(2n-2)$, and the degenerate cases given by Theorem 8.1 of \cite{KL}, all isomorphisms
$$C_{k_1, -1/2} \cong \cW_{\ell}(\gs\go_{2n}, f_{\text{prin}})^{\mathbb{Z}_2},$$ appear on the following list.

$\displaystyle k_1 = -\frac{3 n}{-1 + 2 n},\ k_1 = \frac{7 - 8 n}{2 (2n-1)},\qquad \ell = -(2n-2) + \frac{2n-1}{2 (n-2)},\ \ell = -(2n-2) + \frac{2 (n-2)}{2n-1}$. 
\end{thm}

The proofs of Theorems \ref{thm:c-1/2coinc} and \ref{thm:d-1/2coinc} are again similar to the proof of Theorem \ref{thm:k2spcoin}, and are omitted.

Similarly, we can classify all coincidences between $C_{k_1, k_1}$ and the algebras $\cW_{\ell}(\gs\gp_{2n}, f_{\text{prin}})$ and $\cW_{\ell}(\gs\go_{2n}, f_{\text{prin}})^{\mathbb{Z}_2}$. Again, the arguments are similar to the proof of Theorem \ref{thm:k2spcoin}, and are omitted.

\begin{thm} For $n\geq 2$, aside from the critical levels $k_1 = -2$, $k_1 = -1$ and $\ell = -(n+1)$, and the degenerate cases given by Theorem 8.1 of \cite{KL}, all isomorphisms
$$C_{k_1, k_1} \cong \cW_{\ell}(\gs\gp_{2n}, f_{\text{prin}}),$$ appear on the following list.
\begin{enumerate}

\item $\displaystyle k_1 = n,\qquad \ell = -(n+1) + \frac{n+2}{2 (n+1)}$,

\smallskip

\item $\displaystyle k_1 = -\frac{2n+3}{2 n},\qquad  \ell = -(n+1) +\frac{n}{2 n -3}$,

\smallskip

\item $\displaystyle k_1 = -\frac{2 n}{1 + n}, \qquad \ell = -(n+1) + \frac{n+1}{2 (n-1)}$,

\smallskip

\item $\displaystyle k_1 = -\frac{2 n}{2 n -1} , \qquad \ell = -(n+1) + \frac{n -1}{2 n -1}$,

\smallskip

\item $\displaystyle k_1 = -\frac{4 n}{1 + 2 n}, \qquad \ell = -(n+1) + \frac{2 n -1}{2 (2n+1)}$.
\end{enumerate}
\end{thm}

\begin{remark} In case (1) above, suppose the $k_1 = n = 2m$ is an even integer. Then $\displaystyle \ell = -(2m+1) + \frac{m+1}{2m+1}$ which is not an admissible level for $\gs\gp_{4m}$. However, $\cW_{\ell}(\gs\gp_{4m}, f_{\text{prin}})$ is known to be rational because it is isomorphic to $\cW_{\ell'}(\gs\go_{2m},f_{\text{prin}})$ for $\displaystyle \ell' = -(2n-2) + \frac{2 (n+1)}{2n+1}$, and $\ell'$ is nondegenerate admissible for $\gs\go_{2m}$; see Remark 9.5 of \cite{KL} as well as Theorem \ref{thm:kksocoincidences} below. Therefore $C_{2m, 2m}$ is strongly rational for all positive integers $m$. We also expect that $C_{n,n}$ is strongly rational for all odd positive integers $n$, which is a special case of Conjecture \ref{conj:rationality}. If so, this would imply the strong rationality of $\cW_{\ell}(\gs\gp_{2n}, f_{\text{prin}})$ at the nonadmissible level $\displaystyle \ell = -(n+1) + \frac{n+2}{2 (n+1)}$.
\end{remark}

\begin{thm} \label{thm:kksocoincidences} For $n\geq 2$, aside from the critical levels $k_1 = -2$, $k_1 = -1$ and $\ell = -(2n-2)$, and the degenerate cases given by Theorem 8.1 of \cite{KL}, all isomorphisms
$$C_{k_1, k_1} \cong \cW_{\ell}(\gs\go_{2n}, f_{\text{prin}})^{\mathbb{Z}_2},$$ appear on the following list.
\begin{enumerate}

\item $\displaystyle k_1 = 2n,\qquad \ell = -(2n-2) + \frac{2n+1}{2 (n+1)},\ \ell = -(2n-2) + \frac{2 (n+1)}{2n+1}$,

\smallskip

\item $\displaystyle k_1 = -\frac{4n -3}{2 n}, \qquad \ell = -(2n-2) + \frac{2n-3}{2 n},\ \ell = -(2n-2) + \frac{2 n}{2n-3}$,

\smallskip

\item $\displaystyle k_1 = -\frac{n}{n-1},\qquad \ell = -(2n-2) + \frac{n-1}{n-2},\  \ell = -(2n-2) + \frac{n-2}{n-1}$.

\end{enumerate}
\end{thm}

Next, using the fact that $C^{k_1, -1}$ is isomorphic to the $\mathbb{Z}_2$-orbifold of a one-parameter quotient of the algebra $\cW(c,\lambda)$ constructed in \cite{L6}, we can classify all coincidences between $C_{k_1, -1}$ and other algebras arising as $\mathbb{Z}_2$-orbifolds of quotients of $\cW(c,\lambda)$. For example, the $\mathbb{Z}_2$-orbifolds of type $A$ principal $\cW$-algebras $\mathcal{W}_{\ell}(\mathfrak{sl}_n, f_{\text{prin}})^{\mathbb{Z}_2}$ have this property, and the next result follows immediately from Theorem \ref{thm:onepara-1} and Theorem 11.5 of \cite{L6}.

\begin{thm}\label{thm:Wiso} For $n\geq 3$, aside from the critical levels $k_1 = -2$, $k_1 =-1$, and $\ell = -n$, and the degenerate cases given by Theorem 10.1 of \cite{L6}, all isomorphisms
$C_{k_1, -1}\cong \mathcal{W}_{\ell}(\mathfrak{sl}_n, f_{\text{prin}})^{\mathbb{Z}_2}$ appear on the following list.
\begin{equation} k_1 = -\frac{n+2}{n}, \qquad k_1= -\frac{2 (n-1)}{n}\qquad \ell =  -n + \frac{n-2}{n} ,\qquad \ell = -n + \frac{n}{n-2}.\end{equation} 
\end{thm}

Similarly, in the terminology of \cite{L6}, recall the generalized parafermion algebra 
$$\mathcal{G}^{\ell}(n) = \text{Com}(V^{\ell}(\mathfrak{gl}_n), V^{\ell}(\mathfrak{sl}_{n+1})),$$ and its simple quotient $\mathcal{G}_{\ell}(n)$. By Theorem 8.3 of \cite{L6}, this also arises a quotient of $\mathcal{W}(c,\lambda)$ and the corresponding truncation curve is given explicitly by (8.4) of \cite{L6}. Moreover, the full automorphism group of $\mathcal{G}^{\ell}(n)$ is $\mathbb{Z}_2$ by the same argument as Corollary 5.13 of \cite{L6}. An immediate consequence of Theorem \ref{thm:onepara-1} and Theorem 7.6 of \cite{CLR} is the following.

\begin{thm} \label{genparacoincid} For $n \geq 3$, aside from the critical levels $k_1 = -2$, $k_1 =-1$, $\ell = -n$, and $\ell = -n-1$, and the degenerate cases given by Theorem 10.1 of \cite{L6}, all isomorphisms $C_{k_1, -1} \cong \mathcal{G}_{\ell}(n)^{\mathbb{Z}_2}$ appear on the following list.

\begin{enumerate}
\item  $\displaystyle k_1 =  -\frac{n}{1 + n} ,\qquad  k_1 = - \frac{3 + 2 n}{n+1}, \qquad \ell =  - 2 (1 + n)$,

\smallskip

\item $\displaystyle k_1 = n-3, \qquad k_1 = -n,\qquad \ell =  - 2$,

\smallskip

\item  $\displaystyle k_1= -\frac{n-3}{n},\qquad  k_1 = - \frac{3 + 2 n}{n},\qquad  \ell =  -\frac{2 n}{3}$.
\end{enumerate}
\end{thm}
In the second case one might wonder if $L_{k_1-1}(\gs\gl_2)$ is a  subalgebra of $L_{k_1}(\gs\gl_2) \otimes V^{-1}(\gs\gl_2)$. This is impossible since 
$L_{k_1-1}(\gs\gl_2)$ has a singular vector at conformal weight $k_1$ while the first singular vector of $L_{k_1}(\gs\gl_2) \otimes V^{-1}(\gs\gl_2)$
has conformal weight $k_1+1$.

We remark that in \cite{L6}, conjectural paremetrizations of the trunctation curves for two other infinite families of quotients of $\cW(c,\lambda)$ were written down. For example, the coset $\text{Com}(V^{k+1}(\gg\gl_{n-2}), \cW^k(\gs\gl_n, f_{\text{min}}))$ of the affine vertex subalgebra inside the minimal $\cW$-algebra of $\gs\gl_n$, arises as such a quotient, and its truncation curve is given by Conjecture 9.3 of \cite{L6}. Similarly, the coset $\text{Com}(\cH, \cW^k(\gs\gl_n, f_{\text{subreg}}))$ of the Heisenberg algebra inside the subregular $\cW$-algebra of $\gs\gl_n$, is expected to be such a quotient with truncation curve given by Conjecture 9.6 of \cite{L6}. Assuming these conjectures, the coincidences between $C_{k_1, -1}$ and the simple quotients of the orbifolds $\text{Com}(\cH, \cW^k(\gs\gl_n, f_{\text{subreg}}))^{\mathbb{Z}_2}$ and $\text{Com}(V^{k+1}(\gg\gl_{n-2}), \cW^k(\gs\gl_n, f_{\text{min}}))^{\mathbb{Z}_2}$ can be easily classified, and this is left to the reader.

\subsection{Coincidences between $C_{r,s}$ and $C_{r',s'}$}
By the triality result, $C_{r,s} \cong C_{r, -r-s-4} \cong C_{s, -r-s-4}$, and it is natural to ask whether there exist other coincidences $C_{r,s}$ and $C_{r',s'}$ that do not come from triality. We can find some examples by finding the intersection points on the truncation curves for $C_{k,-2}$, $C_{k,-1/2}$, and $C_{k,k}$. Note that we are regarding $C_{r,s}$ as the simple quotient the specialization of $C^{k_1, k_2}$ to $k_1 = r$ and $k_2 = s$, even if it is a proper subalgebra of the simple quotient of $\text{Com}(V^{k_1+k_2}(\gs\gl_2), V^{k_1}(\gs\gl_2) \otimes V^{k_2}(\gs\gl_2))$. This can only happen if $r+s +2 \in \mathbb{Q}_{\leq 0}$.

\begin{thm} 
\begin{enumerate}
\item The curves $p_2(c,\lambda)$ and $p_3(c,\lambda)$ which realize $C^{k_1}_2$ and $C^{k_1}_{-1/2}$ as quotients of $\cW^{\text{ev}}(c,\lambda)$, respectively, intersect at the point $(c,\lambda) = (15, \frac{221}{9506})$. Using the parametrizations \eqref{para:k2} and \eqref{para:kh}, we obtain the isomorphism
$$C_{-8/3,\ 2} \cong C_{-5/4,\ -1/2}.$$

\item The curves $p_2(c,\lambda)$ and $p_4(c,\lambda)$ which realize $C^{k_1}_2$ and $C^{k_1, k_2}_{J}$ as quotients of $\cW^{\text{ev}}(c,\lambda)$, respectively, intersect at the point $(c,\lambda) = (\frac{27}{20}, \frac{4}{12397})$. Using \eqref{para:k2} and \eqref{para:kk}, we obtain

$$C_{6,\ 2} \cong C_{3,\ 3} .$$

\item The curves $p_3(c,\lambda)$ and $p_4(c,\lambda)$ intersect at the point $(c,\lambda) = (\frac{49}{5}, \frac{20}{781})$. Using \eqref{para:kh} and \eqref{para:kk}, we obtain

$$C_{-7/3,\ -1/2} \cong C_{-7/2, \ -7/2} .$$

\item Using the parametrization  \eqref{para:kk} of $p_4(c,\lambda)$, we see that the point $(c,\lambda) = (\frac{27}{5}, \frac{25}{1078})$ corresponds to both $k_1 = -\frac{3}{4}$ and $k_1 = -6$. We obtain
$$C_{-3/4,\ -3/4} \cong C_{-6, \ -6} .$$

\end{enumerate}
\end{thm}

\appendix

\section{Small and large $N=4$ superconformal vertex algebras} \label{app:first}
For the convenience of the reader, we reproduce here the full OPE algebras of both the small and large $N=4$ superconformal algebra $V(k,a)$, and the small $N=4$ superconformal algebra $V^k$. First, $V(k,a)$ has strong generators $e, f, h, e', f', h', L, G^{\pm\pm}$, where $L$ is a Virasoro field of central charge $c=-6k-3$. These fields satisfy
\begin{equation}\nonumber
\begin{split}
h'(z) h'(w) & \sim  -2((a+1)k+1)(z-w)^{-2}, \qquad
h(z) h(w) \sim -2 \left((a^{-1}+1)k+1\right)(z-w)^{-2},\\
e' (z) f'(w) & \sim - ((a+1)k+1)(z-w)^{-2} + h'(w)(z-w)^{-1}, \\
e(z) f(w) & \sim -  \left(\frac{a+1}{a}k+1\right)(z-w)^{-2} + h(w)(z-w)^{-1},\\
h' (z) e'(w) &\sim 2e'(w)(z-w)^{-1}, \qquad
h(z) e(w) \sim 2e(w)(z-w)^{-1}, \\
h' (z) f'(w) & \sim -2f'(w)(z-w)^{-1}, \qquad
h(z) f(w) \sim  -2f(w)(z-w)^{-1}.
\end{split}
\end{equation}
These act on the odd fields $G^{\pm \pm}$ by
\begin{equation}\nonumber
\begin{split}
h'(z) G^{\pm\pm}(w) & \sim \pm  G^{\pm\pm}(w)(z-w)^{-1}, \qquad
h' (z) G^{\pm\mp}(w) \sim \pm  G^{\pm\mp}(w)(z-w)^{-1}, \\
h (z)  G^{\pm\pm}(w) &\sim \pm  G^{\pm\pm}(w)(z-w)^{-1}, \qquad
h (z)  G^{\pm\mp}(w) \sim \mp  G^{\pm\mp}(w)(z-w)^{-1}, \\
e' (z)  G^{- -}(w) & \sim -  G^{+ -}(w)(z-w)^{-1}, \qquad
e' (z)  G^{- +}(w) \sim -  G^{+ +}(w)(z-w)^{-1}, \\
e (z)  G^{- -}(w) & \sim   G^{- +}(w)(z-w)^{-1}, \qquad
e (z)  G^{+ -}(w) \sim   G^{+ +}(w)(z-w)^{-1}, \\
f' (z)  G^{+ +}(w) & \sim - G^{- +}(w)(z-w)^{-1}, \qquad
f' (z) G^{+ -}(w) \sim  -  G^{- -}(w)(z-w)^{-1}, \\
f (z) G^{+ +}(w) & \sim   G^{+ -}(w)(z-w)^{-1}, \qquad
f (z) G^{- +}(w)  \sim   G^{- -}(w)(z-w)^{-1}, \\
\end{split}
\end{equation}
The OPEs of $G^{\pm, \pm}$ are 
\begin{equation}\nonumber
\begin{split}
G^{+ +}(z)   G^{+ +}(w) &  \sim   \frac{2a}{(a+1)^2} (:e'e:) (w)(z-w)^{-1}, \\
G^{- -} (z) G^{- -}(w) & \sim  \frac{2a}{(a+1)^2} (:f' f:)(w)(z-w)^{-1}, \\ 
G^{- +}(z)  G^{- +}(w) & \sim  -\frac{2a}{(a+1)^2} (:f'e:)(w)(z-w)^{-1},\\
G^{+ -}(z) G^{+ -}(w) & \sim   -\frac{2a}{(a+1)^2} (:e'f:)(w)(z-w)^{-1}, \end{split}
\end{equation}
\begin{equation}\nonumber
\begin{split}
G^{+ +}(z)  G^{- +}(w) & \sim -  \frac{2a}{a+1}\bigg(\frac{1}{a+1} + k \bigg) e(w)(z-w)^{-2} \\ & + \bigg(\frac{a}{(a+1)^2} :h'e: -  \frac{a}{a+1}\bigg(\frac{1}{a+1} + k \bigg) \partial e\bigg)(w)(z-w)^{-1},\\
G^{+ +}(z) G^{+ -}(w) & \sim   \frac{2}{a+1}\bigg(\frac{a}{a+1} + k \bigg) e'(w)(z-w)^{-2} \\ & + \bigg(-\frac{a}{(a+1)^2} :he': +  \frac{1}{a+1}\bigg(\frac{a}{a+1} + k \bigg) \partial e'\bigg)(w)(z-w)^{-1},
\end{split}
\end{equation}
\begin{equation}\nonumber
\begin{split}
G^{- -}(z) G^{- +}(w) & \sim   \frac{2}{a+1}\bigg(\frac{a}{a+1} + k \bigg) f'(w)(z-w)^{-2} \\ & + \bigg(\frac{a}{(a+1)^2} :hf': +  \frac{1}{a+1}\bigg(\frac{a}{a+1} + k \bigg) \partial f'\bigg)(w)(z-w)^{-1},\\
G^{- -}(z)  G^{+ -}(w) & \sim -  \frac{2a}{a+1}\bigg(\frac{1}{a+1} + k \bigg) f(w)(z-w)^{-2} \\ & + \bigg(-\frac{a}{(a+1)^2} :h'f: -  \frac{a}{a+1}\bigg(\frac{1}{a+1} + k \bigg) \partial f\bigg)(w)(z-w)^{-1},\end{split}
\end{equation}
\begin{equation}\nonumber
\begin{split} G^{+ +}(z)  G^{- -}(w)  & \sim   - 2 \left(k(k+1) +\frac{a}{(a+1)^2}\right)(z-w)^{-3} \\
& + \bigg(\frac{a + k + a k}{(1 + a)^2} h' +  \frac{a (1 + k + a k)}{(1 + a)^2} h\bigg)(w)(z-w)^{-2} \\  
& + \bigg(k L +\frac{a}{4 (1 + a)^2} :h' h':  + \frac{a}{4 (1 + a)^2} :hh:  - \frac{a}{2 (1 + a)^2} :h h':   + \frac{a}{(1 + a)^2} :e' f':   \\ 
& + \frac{a}{(1 + a)^2} :ef:  + \frac{a k}{2 (1 + a)} \partial h   + \frac{k}{2 (1 + a)} \partial h' \bigg)(w)(z-w)^{-1},\end{split}
\end{equation}
\begin{equation}\nonumber
\begin{split}
G^{- +}(z)  G^{+ -}(w) &  \sim   2 \left(k(k+1) +\frac{a}{(a+1)^2}\right)(z-w)^{-3} \\
& + \bigg(\frac{a + k + a k}{(1 + a)^2} h' -  \frac{a (1 + k + a k)}{(1 + a)^2} h\bigg)(w)(z-w)^{-2} \\ 
& + \bigg(-k L -\frac{a}{4 (1 + a)^2} :h' h':  - \frac{a}{4 (1 + a)^2} :hh:  - \frac{a}{2 (1 + a)^2} :h h':   - \frac{a}{(1 + a)^2} :e' f':   \\ 
& - \frac{a}{(1 + a)^2} :ef:  - \frac{a k}{2 (1 + a)} \partial h   + \frac{2 a + k + a k}{2 (1 + a)^2} \partial h' \bigg)(w)(z-w)^{-1}\\
\end{split}
\end{equation}

Next, the small $N=4$ superconformal algebra $V^k$ has strong generators $e,f,h,L, G^{\pm \pm}$, where $L$ is a Virasoro element of central charge $c = -6(k+1)$, $e,f,h$ are primary weight one fields which generate a copy of $V^{-k-1}(\gs\gl_2)$, $G^{\pm,\pm}$ are odd primary weight $3/2$ fields, and the following additional OPEs hold:
\begin{equation}\label{eq:small1}\nonumber
\begin{split}
h (z)  G^{\pm\pm}(w) &\sim \pm  G^{\pm\pm}(w)(z-w)^{-1}, \qquad
h (z)  G^{\pm\mp}(w) \sim \mp  G^{\pm\mp}(w)(z-w)^{-1}, \\
e (z)  G^{- -}(w) & \sim   G^{- +}(w)(z-w)^{-1}, \qquad
e (z)  G^{+ -}(w) \sim   G^{+ +}(w)(z-w)^{-1}, \\
f (z) G^{+ +}(w) & \sim   G^{+ -}(w)(z-w)^{-1}, \qquad
f (z) G^{- +}(w)  \sim   G^{- -}(w)(z-w)^{-1}, \\
\end{split}
\end{equation}

\begin{equation} \label{eq:small2} \nonumber
\begin{split}
G^{+ +}(z)   G^{+ +}(w) &  \sim  0, \qquad G^{- -} (z) G^{- -}(w) \sim 0, \\ 
G^{- +}(z) G^{- +}(w) & \sim  0, \qquad G^{+ -}(z) G^{+ -}(w) \sim  0, \\ 
G^{+ +}(z) G^{+ -}(w) & \sim  0,\qquad G^{- -}(z) G^{- +}(w)  \sim  0,\\
G^{+ +}(z)  G^{- +}(w) & \sim -  2k \ e(w)(z-w)^{-2} - k \ \partial e(w)(z-w)^{-1},\\ 
G^{- -}(z)  G^{+ -}(w) & \sim -  2k\ f(w)(z-w)^{-2} - k \ \partial f(w)(z-w)^{-1},
\end{split}
\end{equation}
\begin{equation} \label{eq:small2} \nonumber
\begin{split}
G^{+ +}(z)  G^{- -}(w)  & \sim   - 2 k(k+1)(z-w)^{-3}  + k\ h  (w)(z-w)^{-2} + \big(k L^C  + \frac{k}{2} \partial h \big)(w)(z-w)^{-1},
\\ G^{- +}(z)  G^{+ -}(w) &  \sim   2 k(k+1) (z-w)^{-3}   -  k\ h(w)(z-w)^{-2}  + \big(-k L^C  : - \frac{k}{2} \partial h  \big)(w)(z-w)^{-1}.
\end{split}
\end{equation}

\section{Decoupling relation in $\cH(3)^{\text{SO}_3}$} \label{appendix:decoup}
Here we give the explicit decoupling relation in weight $12$ in $\cH(3)^{\text{SO}_3}$ which is a quantum correction of the classical relation 
$$ (c_{012})^2 - (q_{00} q_{11} q_{22} - q_{00} q_{12} q_{12} - q_{01} q_{01} q_{22} - q_{02} q_{11} q_{02} + q_{01} q_{02} q_{12} + q_{02} q_{12} q_{01}) = 0.$$
\begin{equation} \label{eq:wt12decoup} \begin{split} 
 : C_{012} C_{012}:  & - (:Q_{00} Q_{11} Q_{22}: - :Q_{00} Q_{12} Q_{12}: - :Q_{01} Q_{01} Q_{22}: - :Q_{02} Q_{11} Q_{02}: 
 \\ & + :Q_{01} Q_{02} Q_{12}: + :Q_{02} Q_{12} Q_{01}:) 
 \\ & + \frac{1}{30} :Q_{00} Q_{08}: + \frac{127}{15} :Q_{02} Q_{06}:  - \frac{49}{6} :Q_{04} Q_{04}: - \frac{1}{6} : (\partial^2 Q_{00}) Q_{06}:  
 \\ & + \frac{241}{60} : (\partial Q_{00} )\partial Q_{06}: + \frac{19}{30} :Q_{00} \partial^2 Q_{06}:  + \frac{233}{12} :\partial^2 Q_{02}  Q_{04}:  +  \frac{13}{3} :\partial Q_{02} \partial Q_{04}:  
 \\ & - \frac{409}{12} :Q_{02} \partial^2 Q_{04}:  - \frac{101}{24} :(\partial^4 Q_{00}) Q_{04}:   - \frac{5}{12} :(\partial^3 Q_{00}) \partial Q_{04}:  + 4 :\partial^2 Q_{00} \partial^2 Q_{04}: 
\\ &  - \frac{117}{8} :(\partial Q_{00}) \partial^3 Q_{04}:  - \frac{7}{3} :Q_{00} \partial^4 Q_{04}: + \frac{159}{4} :(\partial^4 Q_{02}) Q_{02}: - \frac{229}{24} :(\partial^3 Q_{02}) \partial Q_{02}:
\\ &  - \frac{17}{2} :(\partial^2 Q_{02}) \partial^2 Q_{02}:  - \frac{159}{20} :(\partial^6 Q_{00}) Q_{02}: + \frac{95}{48} :(\partial^5 Q_{00}) \partial Q_{02}: +  \frac{85}{24}  :(\partial^4 Q_{00}) \partial^2 Q_{02}: 
\\ & + \frac{89}{48} :(\partial^3 Q_{00}) \partial^3 Q_{02}: -  5 :(\partial^2 Q_{00}) \partial^4 Q_{02}: + \frac{4357}{240} :(\partial Q_{00}) \partial^5 Q_{02}: + \frac{14}{5} :Q_{00} \partial^6 Q_{02}: 
\\ &  - \frac{17}{30}  :(\partial^8 Q_{00}) Q_{00}:  - \frac{589}{160}  :(\partial^7 Q_{00}) \partial Q_{00}: + :(\partial^6 Q_{00}) \partial^2 Q_{00}:   - \frac{13}{32} :(\partial^5 Q_{00}) \partial^3 Q_{00}:   
\\ & - \frac{17}{48} :(\partial^4 Q_{00}) \partial^4 Q_{00}: + \frac{313}{450} Q_{0,10}+ \frac{403}{72} \partial^2 Q_{08} -\frac{ 2141}{120} \partial^4 Q_{06}  +  \frac{3653}{45} \partial^6 Q_{04}  \\ & - \frac{1058927}{10080} \partial^8 Q_{02}  + \frac{2156377}{100800} \partial^{10} Q_{00} = 0.    \end{split} 
  \end{equation}

\end{document}